\documentclass[11pt]{article}%

\usepackage{fullpage}

\usepackage{amsfonts}
\usepackage{amsmath}
\usepackage{amssymb}
\usepackage{amsbsy}
\usepackage{amsthm}
\usepackage{epsfig}
\usepackage{graphicx}
\usepackage{colordvi}
\usepackage{graphics}
\usepackage{color}

\numberwithin{equation}{section}

\newtheorem{theorem}{Theorem}[section]

\newtheorem{lemma}{Lemma}[section]

\newtheorem{remark}{Remark}[section]

\newtheorem{algorithm}{Algorithm}[section]

\newcommand{\Real}{\mathbb{R}}

\newcommand{\bu}{{u}}
\newcommand{\bv}{{v}}
\newcommand{\bw}{{w}}

\newcommand{\bn}{{\bf n}}

\newcommand{\bchi}{{\chi}}

\begin{document}

\title{Unconditional long-time stability of a velocity-vorticity method for the 2D Navier-Stokes equations}
\author{
Timo Heister\footnote{Department of Mathematical Sciences, Clemson University, Clemson, SC 29634 (heister@clemson.edu), partially supported by the Computational Infrastructure in Geodynamics initiative (CIG), through the
National Science Foundation under Award No. EAR-0949446 and The University of California -- Davis.}
\and
Maxim A. Olshanskii
\footnote{Department of Mathematics, University of Houston, Houston TX 77004 (molshan@math.uh.edu), partially supported by Army Research Office Grant 65294-MA.}
\and
Leo G. Rebholz\footnote{Department of Mathematical Sciences, Clemson University, Clemson, SC 29634 (rebholz@clemson.edu), partially supported by Army Research Office Grant 65294-MA.}
}
\date{}

\maketitle

\begin{abstract}
We prove unconditional long-time stability for a particular velocity-vorticity discretization of the 2D Navier-Stokes equations.  The scheme begins with a  formulation that uses the Lamb vector to couple the usual velocity-pressure system to the vorticity dynamics equation, and then discretizes with the finite element method in space and implicit-explicit BDF2 in time, with the vorticity equation decoupling at each time step.  We prove the method's vorticity and velocity are both long-time stable in the $L^2$ and $H^1$ norms, without any timestep restriction.  Moreover, our analysis avoids the use of Gronwall-type estimates, which leads us to stability bounds  with only polynomial (instead of exponential) dependence on the Reynolds number.  Numerical experiments are given that demonstrate the effectiveness of the method.
\end{abstract}

\section{Introduction}
The paper addresses long-time stability of numerical methods for the two-dimensional Navier-€"Stokes system describing the motion of   incompressible Newtonian fluids:
\begin{equation}
\begin{split}
\frac{\partial {u}}{\partial t} - \nu \Delta {u} +
({u} \cdot \nabla) {u} + \nabla p &=
{f}, \\
\text{div} \; {u} &= 0,
\end{split}
\label{NS}
\end{equation}
where ${u}={u}({x},t)$ denotes a velocity vector field,  $p=p({x},t)$ is the pressure,  and $f=f({x},t)$ represents
(given) external forcing.
The solution to \eqref{NS} is well-known (see \cite{foias1989gevrey}) to be smooth for all time  in the periodic setting, that is, the domain $\Omega$ is a 2D torus $\mathbb{T}^2$, all functions have mean zero over the torus, and the forcing term $f$ is smooth.
Moreover, the solution of \eqref{NS} is long-time stable,  in the sense that the  norms $\|u\|_{L^2(\Omega)}$ and $\|u\|_{H^1(\Omega)}$ are bounded uniformly in time for $f\in L^\infty(\mathbb{R}_+,L^2(\Omega))$ and initial value $u_0\in H^1(\Omega)$, $\int_\Omega u_0=0$.
The long-time stability is a key  property of \eqref{NS} if one is interested in simulation of a large time scale phenomena or recovering long term statistics, as commonly the case for simulation of flows with large Reynolds' numbers, weather prediction, or climate modeling.  Therefore, it is of practical interest to design numerical methods for \eqref{NS} which inherit this important property. It is also interesting to explore to what extent popular numerical approaches to \eqref{NS} are long-time stable.

The topic of long-time stability and error control for numerical methods for the Navier--Stokes equations
is not new in the literature. Heywood and Rannacher in \cite{heywood1986finite,heywood1990finite} proved uniform in time stability and
error estimate in the \textit{energy norm} for a Crank--Nicolson  Galerkin method applied to 3D Navier-Stokes system, assuming
the solution of the initial boundary value problem is stable. Simo and Armero in \cite{simo1994unconditional}  examined the long-time stability in the energy norm  of several time integration algorithms, including coupled schemes and fractional step/projection methods.  More recent studies include the papers \cite{tone2006long,tone2007long,badia2010long,wang2012efficient,gottlieb2012long}.
The work of Tone and Wirosoetisno \cite{tone2006long,tone2007long} proved uniform in time bounds on $\|u(t_n)\|_{L^2(\Omega)}$ and  $\|\nabla u(t_n)\|_{L^2(\Omega)}$ for implicit Euler and Crank--Nicolson methods. These bounds are subject to restrictions on time step in terms of $\nu$ and a spatial discretization parameter.  Badia et al showed in  \cite{badia2010long} that $\nabla u \in L^\infty (0,\infty;L^2(\Omega))$ for a solution to spatially discretized equations \eqref{NS}.
First and second order semi-explicit time discretization methods for \eqref{NS} written in vorticity--stream function formulation  were studied by X. Wang and co-workers
in \cite{gottlieb2012long,wang2012efficient}. Both papers consider spectral discretization in space, and prove long-time stability bounds for the enstrophy and the $H^1$-norm of the vorticity, again all subject to a time step restriction  of the form $\Delta t\le c\,Re^{-1}$. Thus, despite progress, the current understanding of the long-time behavior of numerical methods for \eqref{NS} is far from being full: only a few studies address uniform in time error estimates for vorticity or velocity gradient, time step restrictions are common in the analyses, and semi-discrete methods are often treated rather than full discretizations. Moreover, to our knowledge, all proofs of long-time numerical stability bounds for vorticity and the gradient of velocity,  invoke a variant of the discrete Gronwall lemma, which results in the dependence of the bounds on the Reynolds number of the form $O(\exp(c^2 Re))$ or even  $O(\exp(c^2 Re^2))$. Although being time independent, such bounds are not very practical for higher   Reynolds number flows; see  \cite{johnson1995numerics} for a discussion and  an effort to improve numerical stability and error estimates dependence on $Re$ number, but only \textit{locally} in time.

In this paper, we prove unconditional long-time stability of a fully discrete numerical method for \eqref{NS}: For $f\in L^\infty (0,\infty;H^1(\Omega))$ we prove uniform in time estimates for the \textit{kinematic energy, enstrophy}, as well as  the $L^2$ norms of \textit{velocity gradient} and \textit{vorticity gradient} of a discrete system.
A finite element method is used for the spatial discretization, and both first and second order time stepping semi-implicit (linear at each time step) schemes are studied. The stability bounds are unconditional, i.e., absolutely no time step restrictions are imposed.
Furthermore, our analysis does not rely on any Gronwall type estimate, which allows us to \textit{avoid exponential dependence of stability bounds on the Reynolds number}. In the present analysis, the dependence is polynomial. Our analysis reveals that the polynomials degree can be significantly lowered at the expense of logarithmic dependence on the spatial mesh size.

The results of the paper systematically exploit the relationship between the vorticity and velocity of the Navier-Stokes system by considering the
vorticity dynamics equation and writing the inertia in the momentum equation in the form of Lamb vector. For $w=\nabla\times u$ and $P=\frac12|u|^2+p$, we reformulate \eqref{NS} as:
 \begin{equation}
\begin{split}
\frac{\partial {u}}{\partial t} - \nu \Delta {u} +
w\times {u} + \nabla P &=
{f}, \\
\text{div} \; {u} &= 0,\\
\frac{\partial {w}}{\partial t} - \nu \Delta {w} +
({u} \cdot \nabla) {w}  &=
\nabla\times{f},
\end{split}
\label{NSw}
\end{equation}
{ where $w\times u:= \left[  -u_2 w, \, u_1 w  \right]^T$.}  Vorticity  plays a fundamental role in fluid dynamics, and
studying properties of \eqref{NS} through the vorticity equation is a well established approach in the Navier-Stokes theory, see, e.g., \cite{majda2002vorticity,gallay2002invariant}. It is also not uncommon in numerical analysis to design numerical methods based on the vorticity equation, e.g.,~\cite{gatski1991review,gunzburger2012finite}. For numerical methods, standard closures for the vorticity equations are obtained either in  vorticity--stream function variables or with the help of the vector  Poisson equation, $\Delta u=-\nabla\times w$. However, recent papers \cite{OR10,LOR11} have demonstrated  numerical advantages of complementing the vorticity equation with the velocity dynamic equation as in \eqref{NSw}. Thus, \eqref{NSw} will be the departure  point in the present analysis.

The rest of the paper is organized as follows. Section \ref{s_prelim} gathers necessary definitions and preliminary results for the analysis that follows.
In Section~\ref{s_Euler}, we introduce a first order time stepping method and prove its long-time stability with respect to the velocity and vorticity $H^1$ norms.
Section~\ref{s_BDF} introduces a second order method based on BDF2 time discretization. We extend the long-time stability results for this method by taking care of some extra technical details. Since the numerical scheme is non-standard, we also provide with our analysis a series of numerical experiments for a 2D flow past a bluff object. The results of the
experiments are presented in Section~\ref{s_exp}, and they illustrate the long-time stability and the performance of the method.

We finish the introduction with the following remark. Most of our stability  analysis  is restricted to the 2D case and, due to the current lack of understanding of the long time behavior of 3D Navier-Stokes solutions, we cannot say to what an extend the results
remain valid in 3D. However, the numerical approach studied here has a straightforward extension to 3D, and relying on a past experience, we believe that numerical methods which are physically consistent and computationally efficient for 2D problems are
commonly found to be also advantageous for solving 3D Navier-Stokes equations.

\section{Notation and Preliminaries}\label{s_prelim}

We consider a domain $\Omega=(0,2\pi)^2\subset \Real^2$, and we restrict this study to the case of periodic boundary conditions.  We note that our stability analysis also holds for the case of full Dirichlet velocity and vorticity boundary conditions.

We use the notation $(\cdot,\cdot)$ and $\| \cdot \|$ for the $L^2(\Omega)$ inner product and norm, respectively.  All other norms will be clearly labeled with subscripts.

The natural velocity and pressure spaces in the periodic setting for the Navier-Stokes equations are
\begin{align*}
X & :=  H^1_{\#}(\Omega)^2 = \{ v\in H^1_{loc}(\mathbb{R})^2,\ v \mbox{ is $2\pi$-periodic in each direction}, \  \int_{\Omega} v \ dx=0\}, \\
Q & :=  L^2_{\#}(\Omega) = \{ q\in L^2_{loc}(\mathbb{R})^2,\ q \mbox{ is $2\pi$-periodic in each direction}, \ \int_{\Omega} q \ dx=0 \}.
\end{align*}
In two dimensions, vorticity is considered as a scalar, and we define vorticity space as
\[
 Y := H^1_{\#}(\Omega) = \{ v\in H^1_{loc}(\mathbb{R}),\ v \mbox{ is $2\pi$-periodic in each direction}, \  \int_{\Omega} v \ dx=0 \}.
 \]

For the discrete setting, we assume $\tau_h$ is a regular, conforming triangulation of $\Omega$ which is compatible with periodic boundary conditions.  Let $(X_h,Q_h)\subset (X,Q)$ be inf-sup stable velocity-pressure finite element spaces, $Y_h\subset Y$ be the discrete vorticity space, all defined as piecewise polynomials on $\tau_h$.

The discretely divergence-free subspace will be denoted by
\[
V_h := \{ v_h \in X_h,\ (\nabla \cdot v_h,q_h)=0\  \forall q_h\in Q_h \}.
\]
The dual space of $V_h$ is denoted by $V_h^*$ with norm $\| \cdot \|_{V_h^*}$.

We will utilize in our analysis discrete analogues of the Laplacian operator.  Define $\Delta_h$ to be the discrete Laplacian operator on $Y_h$:  Given $\phi \in H^1(\Omega)$, $\Delta_h \phi \in Y_h$ satisfies
\[
(\Delta_h \phi,v_h) = -(\nabla \phi,\nabla v_h) \ \ \forall v_h\in Y_h.
\]
Define $A_h$ to be a discretely divergence-free Laplace operator, often referred to as a Stokes operator by:  Given $\phi\in H^1(\Omega)$, $A_h\phi \in V_h$ satisfies
\[
(A_h \phi,v_h) = (\nabla \phi,\nabla v_h) \ \forall v_h\in V_h,
\]
or equivalently,
\[
(A_h \phi,v_h) - (l_h,\nabla\cdot v_h) + (\nabla \cdot A_h\phi,q_h) = (\nabla \phi_h,\nabla v_h) \ \ \forall (v_h,q_h) \in (X_h,Q_h),
\]
{ where $l_h$ is an artificial Lagrange multiplier used so the divergence constraint does not overdetermine the system.} {  Restricted to $X_h$ and $V_h$, respectively, the linear operators $-\Delta_h$ and $A_h$ are self-adjoint and positive definite. In this case, $(-\Delta_h)^{\frac12}:\,X_h\to X_h$ and $A_h^{\frac12}:\,V_h\to V_h$ are well defined and will be used in the paper.}

The Poincare inequality will be used heavily throughout:  there exists $\lambda$, dependent only on  $\Omega$, satisfying
\[
\| \phi \| \le  \lambda \| \nabla \phi \|\quad \forall \phi\in X,
\]
An immediate consequence on the Poincare inequality and the definition of discrete Stokes and Laplace operators is that the following bounds hold
\begin{align*}
\| \nabla v_h \| & \le  \lambda \| A_h  v_h \|\quad \forall v_h \in V_h, \\
\| \nabla z_h \| & \le  \lambda \| \Delta_h  z_h \|\quad\forall z_h \in Y_h.
\end{align*}
We recall  the following discrete Agmon inequalities, which are also consequences of discrete Gagliardo-Nirenberg estimates, see \cite{heywood1982finite} p.298:
\begin{align}
\| v_h \|_{L^{\infty}} &\le C \| v_h \|^{1/2} \| A_h v_h \|^{1/2} \ \forall v_h\in V_h, \label{agmon} \\
\| z_h \|_{L^{\infty}} &\le C \| z_h \|^{1/2} \| \Delta_h z_h \|^{1/2} \ \forall z_h\in Y_h, \label{agmon2},
\end{align}
where $C$ is independent of $h$.  The discrete Sobolev inequality (proven in \cite{GG08}),
\begin{equation}
\| \nabla \phi_h \|_{L^4} \le { C \| \nabla \phi_h \|^{1/2} \| \Delta_h \phi_h \|^{1/2} } \ \forall \phi_h\in X_h, \label{discreteSobolev}
\end{equation}
again with $C$ independent of $h$, allows us to prove the following lemma.
\begin{lemma}
For every $z_h\in Y_h$, there exists a constant $C$, independent of $h$, satisfying
\begin{align}
\| \nabla z_h \|_{L^3} \le C \| z_h \|^{1/3} \| \Delta_h z_h \|^{2/3} \ \forall z_h\in Y_h.  \label{agmon3}
\end{align}
\end{lemma}
\begin{proof}
By H\"older's inequality,
\[
\| \nabla z_h \|^3_{L^3} \le \| \nabla z_h \| \| \nabla z_h \|^2_{L^4},
\]
and thus using \eqref{discreteSobolev} provides the bound
\[
\| \nabla z_h \|^3_{L^3} \le { C} \| \nabla z_h \|^2 \| \Delta_h z_h \|.
\]
Since $\| \nabla z_h \|^2 = (\nabla z_h,\nabla z_h)=-(z_h,\Delta_h z_h) \le \| z_h \| \| \Delta_h z_h \|$, the estimate becomes
\[
\| \nabla z_h \|^3_{L^3} \le { C} \| z_h \| \| \Delta_h z_h \|^2.
\]
Taking cube roots of both sides completes the proof.
\end{proof}

Define the skew-symmetric trilinear operator $b^*: X_h\times Y_h \times Y_h \rightarrow \Real$ by
\[
b^*(u,w,\chi)=(u\cdot\nabla w,\chi) + \frac12 ((\nabla \cdot u)w,\chi).
\]
We will exploit the property that $b^*(u,w,w)=0$ in our analysis of the vorticity equation.

\section{Backward Euler}\label{s_Euler}

We first consider long-time stability of the velocity-vorticity scheme with finite element spatial discretization and backward Euler temporal discretization.  The algorithm decouples the vorticity equation by using a first order approximation of the vorticity in the momentum equation, and reads as follows.

\begin{algorithm} \label{bealg} \rm
Given the forcing $f$ and initial velocity $u_0$, set $u_h^0$ to be the interpolant of $u_0$, and $w_h^0$ the interpolant of the curl of $u_0$.  Select a timestep $\Delta t>0$, and for n=0,1,2,...\\
Step 1:  Find $(u_h^{n+1},{ P}_h^{n+1})\in (X_h,Q_h)$ satisfy for every $(v_h,q_h)\in (X_h,Q_h)$,
\begin{align}
\frac{1}{\Delta t}\left( u_h^{n+1} - u_h^n ,v_h\right) + ( w_h^n  \times u_h^{n+1},v_h) 
- ({ P}_h^{n+1},\nabla \cdot v_h) + \nu(\nabla u_h^{n+1},\nabla v_h) & =  (f^{n+1},v_h). \label{bescheme1} \\
(\nabla \cdot u_h^{n+1},q_h) & =  0, \label{bescheme2}
\end{align}
Step 2: Find $w_h^{n+1}\in Y_h$ satisfy for every $\chi_h\in Y_h$,
\begin{equation}
\frac{1}{\Delta t}\left( w_h^{n+1} - w_h^n,\chi_h\right) + b^*( u_h^{n+1} , w_h^{n+1},\chi_h) + \nu(\nabla w_h^{n+1},\nabla \chi_h) =  (\nabla \times f^{n+1},\chi_h). \label{bescheme3}
\end{equation}
\end{algorithm}
{ We note that $P_h^{n}$ represents Bernoulli pressure, and thus is intended to approximate $\frac12 | u(t^{n}) |^2 + p(t^{n})$.  To recover a zero-mean approximation to the kinematic pressure, one can rescale $p_h^{n}:=P_h^n - \frac12 |u_h^n|^2$ accordingly.}

We will prove long-time $L^2$ and $H^1$ stability of both the velocity and the vorticity.  We begin with the $L^2$ results.

\begin{theorem}[Long-time $L^2$ stability of velocity and vorticity]\label{thm1}
Suppose $f\in L^{\infty}(0,\infty;L^2(\Omega))$, and $u_0 \in H^1(\Omega)$. Denote $\alpha:=(1+\nu{ \lambda^{-2}}\Delta t)$.  For any $\Delta t>0$, we have that solutions of Algorithm \ref{bealg} satisfy for every positive integer $n$,
\begin{align}\label{est_u1}
 \| u_h^{n} \|^2+\frac{\nu{ \Delta t}}{2} \sum_{k=0}^{n-1}\left( \frac{1}{\alpha}\right)^{n-k}\|\nabla u_h^{k+1}\|^2 & \le    \left( \frac{1}{\alpha}\right)^{n} \| u_h^0 \|^2  +  \frac{2\alpha { \lambda^{2}}}{\nu^2}\,\| f \|_{L^{\infty}(0,\infty;V_h^*)}^2   =: C_0^2, \\ \label{est_w1}
 \| w_h^{n} \|^2+\frac{\nu{ \Delta t}}{2} \sum_{k=0}^{n-1}\left( \frac{1}{\alpha}\right)^{n-k}\|\nabla w_h^{k+1}\|^2 & \le    \left( \frac{1}{\alpha}\right)^{n} \| w_h^0 \|^2  +   \frac{2\alpha { \lambda^{2}}}{\nu^2 }\,\| f \|_{L^{\infty}(0,\infty;L^2(\Omega))}^2   =: C_1^2,
\end{align}
\end{theorem}
\begin{remark}\rm
The constants $C_0$ and $C_1$ are independent of $n$ and therefore hold for arbitrarily large $n$.   These bounds can be considered as  dependent only on the data (since time step sizes are inherently bounded above), and moreover, for sufficiently large $n$ the bounds are independent of the initial condition.
\end{remark}

\begin{proof}
Take $v_h=2\Delta t u_h^{n+1}$, $q_h={P}_h^{n+1}$, and $\chi_h=2 \Delta t w_h^{n+1}$, which vanishes the nonlinear and pressure terms, and leaves
\begin{align*}
\| u_h^{n+1} \|^2 - \| u_h^n \|^2 + \| u_h^{n+1} - u_h^n \|^2  + 2\Delta t \nu \| \nabla u_h^{n+1} \|^2 & = 2\Delta t (f^{n+1},u_h^{n+1}), \\
 \| w_h^{n+1} \|^2 - \| w_h^n \|^2 + \| w_h^{n+1} - w_h^n \|^2  + 2\Delta t \nu \| \nabla w_h^{n+1} \|^2 & = 2\Delta t (\nabla \times f^{n+1},w_h^{n+1}).
\end{align*}
We majorize the forcing terms after integrating by parts in the vorticity equation forcing term, applying Young's inequality, and dropping positive terms on the left hand sides to get
\begin{align*}
 \| u_h^{n+1} \|^2 - \| u_h^n \|^2 + \frac32\nu\Delta t \| \nabla u_h^{n+1} \|^2 & \le  2\nu^{-1}\Delta t \| f^{n+1} \|_{V_h^*}^2, \\
\| w_h^{n+1} \|^2 - \| w_h^n \|^2  + \frac32\nu\Delta t \| \nabla w_h^{n+1} \|^2 & \le  2\nu^{-1} \Delta t \|  f^{n+1} \|^2.
\end{align*}
From here, the velocity and vorticity estimates follow identically, except that the norm on the forcing term is different, and thus we restrict the remainder of the proof to only the velocity.  Applying the Poincare inequality to lower bound the viscous term yields
\begin{align*}
(1+\nu { \lambda^{-2}}\Delta t) \| u_h^{n+1} \|^2 + \frac\nu2\Delta t \| \nabla u_h^{n+1} \|^2 & \le  \| u_h^n \|^2   + 2\nu^{-1}\Delta t \| f^{n+1} \|_{V_h^*}^2.
\end{align*}
Now fix an integer $N>0$ and divide the above inequality by $\alpha^{N-n}$ to obtain
\begin{equation*}
 \left(\frac1\alpha\right)^{N-n-1}\| u_h^{n+1} \|^2 +  \left(\frac1\alpha\right)^{N-n} \frac\nu2\Delta t \| \nabla u_h^{n+1} \|^2  \le    \left(\frac1\alpha\right)^{N-n}\| u_h^n \|^2   +   \left(\frac1\alpha\right)^{N-n} 2\nu^{-1}\Delta t \| f^{n+1} \|_{V_h^*}^2.
\end{equation*}
Summing up for $n=0,\dots,N-1$ and reducing, we get
\begin{align*}
 \| u_h^{N} \|^2 +\frac{\nu{ \Delta t}}{2} \sum_{n=0}^{N-1}\left( \frac{1}{\alpha}\right)^{N-n}\|\nabla u_h^{n+1}\|^2 & \le \left( \frac{1}{\alpha}\right)^{N} \| u_h^0 \|^2   + 2\nu^{-1}\Delta t \| f \|_{L^{\infty}(0,\infty;V_h^*)}^2   \sum_{n=0}^{N-1} \left( \frac{1}{\alpha} \right)^{N-n} \\
  & \le
    \left( \frac{1}{\alpha}\right)^{N} \| u_h^0 \|^2   + 2\nu^{-1}\Delta t \| f \|_{L^{\infty}(0,\infty;V_h^*)}^2  \frac{ \alpha}{\alpha - 1 }.
\end{align*}
Substituting for $\alpha$ proves the velocity result.  Applying the same steps for vorticity produces estimate \eqref{est_w1}, which finishes the proof of the theorem.
\end{proof}

\begin{theorem}[Long-time $H^1$ stability of velocity] \label{thm2}
Suppose $f\in L^{\infty}(0,\infty;L^2(\Omega))$, and $u_0 \in H^1(\Omega)$. Denote $\alpha:=(1+\nu { \lambda^{-2}}\Delta t)$.  For any $\Delta t>0$, the solutions of Algorithm \ref{bealg} satisfy for every positive integer $n$,
\begin{equation}\label{est_u2}
 \| \nabla u_h^{n} \|^2  \le     \left( \frac{1}{\alpha}\right)^{n} \| \nabla u_h^0 \|^2  + \left(  2\nu^{-1} \| f \|_{L^{\infty}(0,\infty;L^2)}^2 + C \nu^{-3}C_1^4 C_0^2\right)  \frac{\alpha{ \lambda^{2}}}{\nu  } =: C_2^2.
\end{equation}
and
\begin{equation}\label{est_u2a}
 \| \nabla u_h^{n} \|^2  \le     \left( \frac{1}{\alpha}\right)^{n} \| \nabla u_h^0 \|^2  +   2\frac{\alpha{ \lambda^{2}}}{\nu^2  } \| f \|_{L^{\infty}(0,\infty;L^2)}^2 + C{ (1+|\ln h|)}\nu^{-2} C_1^2 C_0^2 =: \widetilde{C}_2^2.
\end{equation}
where $C$ is a generic constant, which depends on Sobolev's embedding inequalities optimal constants and constants from Agmon's type inequalities \eqref{agmon}--\eqref{agmon3}.
\end{theorem}
\begin{remark}\rm
The theorem above proves that the long-time  velocity solution is bounded in the $H^1$ norm only by the problem data, and similar to the $L^2$ bound, it is independent of the initial condition when $n$ is sufficiently large.  \\
With respect to the dependence on Re, the estimate \eqref{est_u2} gives $\| \nabla u_h^{n} \|\le O(Re^5)$,
while estimate \eqref{est_u2a} gives $\| \nabla u_h^{n} \|\le O({ (1+|\ln h|)}^{\frac12} Re^3)$.
\end{remark}

\begin{proof}
Take $v_h=2\Delta t A_h u_h^{n+1}$ in \eqref{bescheme1} to obtain
\[
\| \nabla u_h^{n+1} \|^2 - \| \nabla u_h^n \|^2 + \frac32\nu\Delta t \| A_h u_h^{n+1} \|^2 \le 2\nu^{-1}\Delta t \| f^{n+1} \|^2 + 2\Delta t | \left( w_h^n \times u_h^{n+1},A_h u_h^{n+1} \right) | .
\]
For the last term on the right-hand side, we majorize it first using Holder's inequality, the discrete Agmon inequality \eqref{agmon}, Young's inequality, and Theorem \ref{thm1} to find
\begin{align*}
 | \left( w_h^n \times u_h^{n+1},A_h u_h^{n+1} \right) |
 & \le  \| w_h^n \| \| u_h^{n+1} \|_{L^{\infty}} \| A_h u_h^{n+1} \| \\
 & \le  C \| w_h^n \| \| u_h^{n+1} \|^{1/2}  \| A_h u_h^{n+1} \|^{3/2} \\
 & \le  C\nu^{-3} \| w_h^n \|^4 \| u_h^{n+1} \|^{2} + \frac{\nu}{2}  \| A_h u_h^{n+1} \|^{2} \\
 & \le  C\nu^{-3}C_1^4 C_0^{2} + \frac{\nu}{2}  \| A_h u_h^{n+1} \|^{2}.
\end{align*}
Combining these last two inequalities produces
\[
\| \nabla u_h^{n+1} \|^2 - \| \nabla u_h^n \|^2 + \nu\Delta t \| A_h u_h^{n+1} \|^2 \le 2\nu^{-1}\Delta t \| f^{n+1} \|^2 + C\Delta t \nu^{-3} C_1^4 C_0^2,
\]
and thanks to Poincare, we obtain
\begin{align*}
\left(1+ \nu { \lambda^{-2}}\Delta t \right)\| \nabla u_h^{n+1} \|^2
& \le  \| \nabla u_h^n \|^2 + \Delta t \left(  2\nu^{-1} \| f \|_{L^{\infty}(0,\infty;L^2)}^2 + C \nu^{-3} C_1^4 C_0^2  \right).
\end{align*}
Recalling the notation $\alpha = \left(1+ \nu { \lambda^{-2}}\Delta t \right)$, this relation can be written as
\begin{equation}\label{aux1}
\| \nabla u_h^{n+1} \|^2   \le  \frac{1}{\alpha} \| \nabla u_h^n \|^2 +\frac{1}{\alpha} \Delta t\,\left(  2\nu^{-1} \| f \|_{L^{\infty}(0,\infty;L^2)}^2 + C \nu^{-3} C_1^4 C_0^2  \right).
\end{equation}
Recursive substitution and an estimate for the partial sum of a geometric progression lead us to \eqref{est_u2}.

Alternatively, we can employ the finite element inverse inequality { $\|\bu_h\|_{L^\infty(\Omega)}\le C(1+|\ln h|)^{\frac12}\|\nabla\bu_h\|$, valid in 2D (see p.124 in \cite{BS08})}, and estimate the nonlinear terms in the different way:
\begin{align*}
 | \left( w_h^n \times u_h^{n+1},A_h u_h^{n+1} \right) |
 & \le  \| w_h^n \| \| u_h^{n+1} \|_{L^{\infty}} \| A_h u_h^{n+1} \| \\
 & \le  C { (1+|\ln h|)}^{\frac12}\| w_h^n \| \| \nabla u_h^{n+1} \|  \| A_h u_h^{n+1} \| \\
 & \le C{ (1+|\ln h|)}\nu^{-1} C_1^2 \| \nabla u_h^{n+1} \|^{2} + \frac{\nu}{2}  \| A_h u_h^{n+1} \|^{2}.
\end{align*}
Similar arguments that produced \eqref{aux1} give
\begin{equation*}
\| \nabla u_h^{n+1} \|^2   \le  \frac{1}{\alpha} \| \nabla u_h^n \|^2 +\frac{1}{\alpha} \Delta t\,\left(  2\nu^{-1} \| f \|_{L^{\infty}(0,\infty;L^2)}^2 + C{ (1+|\ln h|)} \nu^{-1} C_1^2 \|\nabla u_h^{n+1} \|^{2}  \right).
\end{equation*}
Doing recursive substitution and employing \eqref{est_u1} to estimate the resulting sum $\sum_{k=1}^{n+1}\alpha^{k-n-1}  \|\nabla u_h^{k} \|^{2}$ leads to \eqref{est_u2a}.
\end{proof}

\begin{theorem}[Long-time $H^1$ stability of vorticity] \label{thm3}
Suppose $f\in L^{\infty}(0,\infty;H^1(\Omega))$, and $u_0 \in H^2(\Omega)$. Let $\alpha:=(1+\nu { \lambda^{-2}}\Delta t)$. For any $\Delta t>0$, solutions of Algorithm \ref{bealg} satisfy for every positive integer $n$,
\begin{equation}\label{est_w2}
 \| \nabla w_h^{n} \|^2 \le
      \left( \frac{1}{\alpha}\right)^{n} \| \nabla w_h^0 \|^2
 + \left(   2\nu^{-1} \| f \|_{L^{\infty}(0,\infty;H^1(\Omega))} + \nu^{-5} C_2^6 C_1^2 + C\nu^{-3} C_2^4 C_1^2      \right)  \frac{\alpha{ \lambda^{2}}}{\nu },
 \end{equation}
and
\begin{equation}\label{est_w2a}
 \| \nabla w_h^{n} \|^2 \le
      \left( \frac{1}{\alpha}\right)^{n} \| \nabla w_h^0 \|^2
 + \frac{2\alpha{ \lambda^{2}}}{\nu^2 } \| f \|_{L^{\infty}(0,\infty;H^1(\Omega))} + C { (1+|\ln h|)}\nu^{-2} \widetilde{C}_2^2 C_1^2.
 \end{equation}
\end{theorem}
\begin{remark}\rm
The theorem above proves that the long-time vorticity solution is bounded in the $H^1$ norm only by the problem data, and similar to the $L^2$ bound, it is independent of the initial condition when $n$ is sufficiently large.\\
{Using $C_0=O(Re),\ C_1=O(Re)$ from Theorem \ref{thm1}, and $C_2=O(Re^5)$ from Theorem \ref{thm2}, 
the estimate \eqref{est_w2} gives $\| \nabla w_h^{n} \|\le O(Re^{19})$.  Estimate \eqref{est_w2a} gives $\| \nabla w_h^{n} \|\le O({ (1+|\ln h|)} Re^5)$, using additionally that $\widetilde{C}_2=O((1+|\ln h|)^{\frac12} Re^3)$ from Theorem 2.}
\end{remark}
\begin{proof}
Take $\chi_h=2\Delta t \Delta_h w_h^{n+1}$ in \eqref{bescheme3}, and majorize the forcing term using Cauchy-Schwarz and Young's inequalities to obtain
\begin{equation*}
 \| \nabla w_h^{n+1} \|^2 - \| \nabla w_h^n \|^2  + \frac32\nu\Delta t \| \Delta_h w_h^{n+1} \|^2 \le 2\nu^{-1}\Delta t \| f^{n+1} \|_{H^1}^2
 +2\Delta t  | b^*(u_h^{n+1},w_h^{n+1},\Delta_h w_h^{n+1}) |.
\end{equation*}
We bound the nonlinear term using Holder, Sobolev embeddings, discrete Agmon \eqref{agmon2} and discrete Sobolev inequality \eqref{agmon3}, and Theorems \ref{thm1} and \ref{thm2} to reveal
\begin{align*}
& \hspace{-0.5in} | b^*(u_h^{n+1},w_h^{n+1},\Delta_h w_h^{n+1}) | \\
& \le    | (u_h^{n+1} \cdot \nabla w_h^{n+1},\Delta_h w_h^{n+1}) |+  \frac12 | ( (\nabla \cdot u_h^{n+1}) w_h^{n+1},\Delta_h w_h^{n+1}) | \\
& \le    \|  u_h^{n+1}\|_{L^6} \| \nabla w_h^{n+1} \|_{L^3} \| \Delta_h w_h^{n+1} \| +  \frac12 \| \nabla u_h^{n+1} \| \|  w_h^{n+1} \|_{L^{\infty}} \| \Delta_h w_h^{n+1} \| \\
& \le   C C_2 \|  w_h^{n+1} \|^{1/3}  \| \Delta_h w_h^{n+1} \|^{5/3} + C C_2 \|  w_h^{n+1} \|^{1/2} \| \Delta_h w_h^{n+1} \|^{3/2} \\
& \le   C C_2 C_1^{1/3}  \| \Delta_h w_h^{n+1} \|^{5/3} + C C_2 C_1^{1/2} \| \Delta_h w_h^{n+1} \|^{3/2}.
\end{align*}
The generalized Young's inequality now provides the bound
\[
 | b^*(u_h^{n+1},w_h^{n+1},\Delta_h w_h^{n+1}) |
 \le
C\nu^{-5} C_2^6 C_1^2 + C\nu^{-3} C_2^4 C_1^2 + \frac{\nu}{4} \| \Delta_h w_h^{n+1} \|^2.
 \]
Combining the estimates above yields
\begin{equation*}
 \| \nabla w_h^{n+1} \|^2   + \nu\Delta t \| \Delta_h w_h^{n+1} \|^2 \le  \| \nabla w_h^n \|^2 + C\Delta t \left( \nu^{-1} \| f^{n+1} \|_{H^1} + \nu^{-5} C_2^6 C_1^2 + C\nu^{-3} C_2^4 C_1^2\right),
\end{equation*}
and after applying Poincare we get
\begin{equation*}
\left( 1+ { \lambda^{-2}}\nu\Delta t \right)  \| \nabla w_h^{n+1} \|^2   \le  \| \nabla w_h^n \|^2 + C\Delta t \left( \nu^{-1} \| f^{n+1} \|_{H^1} + \nu^{-5} C_2^6 C_1^2 + C\nu^{-3} C_2^4 C_1^2\right).
\end{equation*}
The remainder of the proof of \eqref{est_w2} follows analogous to the $H^1$ case for velocity.

Alternatively, we may bound the nonlinear terms as follows:
\begin{align*}
& \hspace{-0.5in} | b^*(u_h^{n+1},w_h^{n+1},\Delta_h w_h^{n+1}) | \\
& \le    | (u_h^{n+1} \cdot \nabla w_h^{n+1},\Delta_h w_h^{n+1}) |+  \frac12 | ( (\nabla \cdot u_h^{n+1}) w_h^{n+1},\Delta_h w_h^{n+1}) | \\
& \le    \|  u_h^{n+1}\|_{L^\infty} \| \nabla w_h^{n+1} \| \| \Delta_h w_h^{n+1} \| +  \frac12 \| \nabla u_h^{n+1} \| \|  w_h^{n+1} \|_{L^{\infty}} \| \Delta_h w_h^{n+1} \| \\
& \le   C{ (1+|\ln h|)}^{\frac12}\|\nabla u_h^{n+1}\| \| \nabla w_h^{n+1} \| \| \Delta_h w_h^{n+1} \| \le   C{ (1+|\ln h|)}^{\frac12} \widetilde{C}_2 \| \nabla w_h^{n+1} \| \| \Delta_h w_h^{n+1} \|\\
& \le   C \nu^{-1} { (1+|\ln h|)} \widetilde{C}_2^2 \| \nabla w_h^{n+1} \|^2 +  \frac{\nu}{4} \| \Delta_h w_h^{n+1} \|^2.
\end{align*}
To complete the proof of \eqref{est_w2a} we proceed as above and employ estimate \eqref{est_w1} for
the weighted sum of $\| \nabla w_h^{n+1} \|^2$ norms.
\end{proof}

\section{Second-order method}\label{s_BDF}

We consider next a velocity-vorticity scheme with BDF2 timestepping. The scheme decouples the update of velocity and vorticity on each time step. Similar to the backward Euler case, we shall prove that the velocity and vorticity are both unconditionally long-time stable in both the $L^2$ and $H^1$ norms, and the scalings of the stability estimates with $Re$ are the same as those from the backward Euler analysis.  However, the analysis is somewhat more technical here, and  a special norm is used to handle the time derivative terms.

\begin{algorithm} \label{bdf2}
Given the forcing $f$ and initial velocity $u_0$, set $u_h^{-1}=u_h^0$ to be the interpolant of $u_0$, and $w_h^{-1}=w_h^0$ the interpolant of the curl of $u_0$.  Select a timestep $\Delta t>0$, and for n=0,1,2,...\\
Step 1:  Find $(u_h^{n+1},{ P}_h^{n+1})\in (X_h,Q_h)$ satisfy for every $(v_h,q_h)\in (X_h,Q_h)$,
\begin{align}
\frac{1}{2\Delta t}\left( 3u_h^{n+1} - 4u_h^n + u_h^{n-1},v_h\right) + ( \left(2w_h^n - w_h^{n-1}\right) \times u_h^{n+1},v_h) & \nonumber \\
- ({ P}_h^{n+1},\nabla \cdot v_h) + \nu(\nabla u_h^{n+1},\nabla v_h) & =  (f^{n+1},v_h), \label{scheme1} \\
(\nabla \cdot u_h^{n+1},q_h) & =  0. \label{scheme2}
\end{align}
Step 2: Find $w_h^{n+1}\in Y_h$ satisfy for every $\chi_h\in Y_h$,
\begin{align*}
\frac{1}{2\Delta t}\left( 3w_h^{n+1} - 4w_h^n + w_h^{n-1},v_h\right)  + b^*( u_h^{n+1} , w_h^{n+1},\chi_h) + \nu(\nabla w_h^{n+1},\nabla \chi_h)  =  (\nabla \times f^{n+1},\chi_h). 
\end{align*}
\end{algorithm}

For the $2\times2$ matrix
\[
G := \left( \begin{array}{cc} 1/2 & -1 \\ -1 & 5/2 \end{array} \right),
\]
we introduce the  $G$-norm $\| \chi \|_{G} = (\chi,G\chi)$, $\chi$ is vector valued.
The $G$-norm is widely used in BDF2 analysis, see e.g. \cite{CGSW13,HW02}.
 The following property of the $G$-norm is well-known \cite{HW02}, { however for completeness
 we will provide a short proof.
 \begin{lemma}
Set $\chi^0 = [v^0,\ v^1]^T$ and $\chi^1 = [v^1,\ v^2]^T$.  Then
\[
\left( \frac32 v^2 -2v^1+ \frac12 v^0,v^2 \right) = \frac12 \left( \| \chi^1 \|_G^2 - \| \chi^0 \|_G^2 \right) + \frac{1}{4}\| v^2 - 2v^1 + v^0 \|^2.
\]
\end{lemma}
\begin{proof}
Noting the algebraic identity $a(3a-4b+c)=\frac12\left( (a^2-b^2) + (2a-b)^2-(2b-c)^2 + (a-2b+c)^2\right)$, we can write
\begin{eqnarray*}
\frac12\left( 3v^2 - 4v^1 + v^0,v^2 \right) & = & \frac12 \left( \frac{\| v^2 \|^2 - \|v^1\|^2}{2} + \frac{ \| 2v^2 - v^1 \|^2 - \|2v^1 - v^0\|^2}{2} \right) + \frac{1}{4} \| v^2 - 2v^1 + v^0 \|^2  \\
& = & \frac12 \left( \frac{\| v^2 \|^2 + \| 2v^2 - v^1 \|^2}{2} - \frac{  \|2v^1 - v^0\|^2 + \|v^1\|^2 }{2} \right) + \frac{1}{4}\| v^2 - 2v^1 + v^0 \|^2.
\end{eqnarray*}
The term $\| \chi^0 \|_G^2$ can be decomposed as
\[
\| \chi^0 \|_G^2=(\chi^0,G\chi^0)
=
\left( \left( \begin{array}{c} v^0 \\ v^1 \end{array} \right), \left( \begin{array}{cc} 1/2 & -1 \\ -1 & 5/2 \end{array} \right)\left( \begin{array}{c} v^0 \\ v^1 \end{array} \right) \right),
\]
and with some arithmetic we get that
\[
\| \chi^0 \|_G^2
=
\int_{\Omega} \left( v^0 \left(\frac12 v^0 - v^1\right) + v^1\left(-v^0 + \frac52 v^1\right) \right)\ dx = \frac12 \| v^1 \|^2 + \frac12 \| 2v^1 - v^0 \|^2.
\]
A similar decomposition of $\| \chi^1\|$ completes the proof.
\end{proof}
}

It is also known that the $G$ norm is equivalent to the $L^2(\Omega)$ norm in the sense of there existing $C_l$ and $C_u$ such that
\[
C_l \| \chi \|_G \le \| \chi \| \le C_u \| \chi \|_G.
\]
Use of the $G$-norm and this norm equivalence will allow for a smoother analysis.

We begin our analysis with the long-time $L^2$ stability of velocity and vorticity.

\begin{theorem}[Long-time $L^2$ stability of velocity and vorticity] \label{LTL2}
Let $f\in L^{\infty}(0,\infty;V_h^{*})$ and $u_0\in H^1(\Omega)$.  Then for any $\Delta t>0$, solutions of Algorithm \ref{bdf2} satisfy for every positive integer $n$,
\begin{multline}
C_l^2 \left(  \| u_h^{n} \|^2 + \| u_h^{n-1} \|^2 \right)  + \frac{\nu\Delta t}{4} \| \nabla u_h^{n} \|^2
  \le
\left( \frac{1}{1+\alpha} \right)^n \left(  2C_u \| u_h^0 \|^2+ \frac{\nu\Delta t}{4}\| \nabla u_h^0 \|^2 \right)\\
+
\max \left(  2\Delta t, \frac{4}{\nu C_l^2} \right)
 \nu^{-1}  \| f \|_{L^{\infty}(0,\infty;V_h^*)}^2 =: C_4. \label{bdf2est1}
\end{multline}

If additionally $f\in L^{\infty}(0,\infty;L^2(\Omega))$ and $w_h^0 \in H^1(\Omega)$, then for any $\Delta t>0$, solutions of Algorithm \ref{bdf2} satisfy for every positive integer $n$,
\begin{multline}
C_l^2 \left(  \| w_h^{n} \|^2 + \| w_h^{n-1} \|^2 \right)  + \frac{\nu\Delta t}{4} \| \nabla w_h^{n} \|^2 \le
\left( \frac{1}{1+\alpha} \right)^n \left(  2C_u \| w_h^0 \|^2 + \frac{\nu\Delta t}{4}\| \nabla w_h^0 \|^2 \right)\\
+
\max \left(  2\Delta t, \frac{4}{\nu C_l^2} \right)
 \nu^{-1}  \| f \|_{L^{\infty}(0,\infty;L^2(\Omega))}^2 =: C_5. \label{bdf2est2}
\end{multline}
\end{theorem}
\begin{remark}\rm
A more technical analysis can be made, similar to the backward Euler case, that includes the terms $\frac{\nu}{2} \sum_{k=0}^{n-1} \alpha^{k-n} \| \nabla u_h^{k+1} \|^2$ and $\frac{\nu}{2} \sum_{k=0}^{n-1} \alpha^{k-n} \| \nabla w_h^{k+1} \|^2$ on the left hand sides of \eqref{bdf2est1} and \eqref{bdf2est2}, respectively.
\end{remark}
\begin{proof}
Choose $v_h=2\Delta t u_h^{n+1}$ in \eqref{scheme1}, which vanishes the nonlinear and pressure terms, and then upper bound the forcing term just as in the backward Euler case to get
\begin{equation}
 \| \chi^{n+1} \|_G^2 - \| \chi^n \|_G^2
+ \frac{1}{2} \| u_h^{n+1} - 2u_h^n + u_h^{n-1} \|^2
+ \nu\Delta t \| \nabla u_h^{n+1} \|^2
\le
\nu^{-1} \Delta t \| f^{n+1} \|_{V_h^*}^2,  \label{a1}
\end{equation}
where  $\chi^{n+1} = [u_h^{n},\ u_h^{n+1}]^T$ and $\chi^{n} = [u_h^{n-1},\ u_h^{n}]^T$.  Dropping the { third} term on the left-hand side, and adding $\frac{\nu\Delta t}{4}\| \nabla u_h^n \|^2$ to both sides produces
\begin{multline}
\left(  \| \chi^{n+1} \|_G^2  + \frac{\nu\Delta t}{4} \| \nabla u_h^{n+1} \|^2 \right)
 + \frac{\nu\Delta t}{4}\left( \| \nabla u_h^{n+1} \|^2 + \| \nabla u_h^n \|^2 \right)
 + \frac{\nu\Delta t}{2} \| \nabla u_h^{n+1} \|^2
 \\ \le
\left( \| \chi^n \|_G^2 + \frac{\nu\Delta t}{4}\| \nabla u_h^n \|^2 \right)
+ \nu^{-1} \Delta t \| f \|_{L^{\infty}(0,\infty;V_h^*)}^2.  \label{a2}
\end{multline}
Using the Poincare inequality and then the equivalence of the $G$-norm with the $L^2$ norm, we have that
\begin{equation*}
 \frac{\nu\Delta t}{4}\left( \| \nabla u_h^{n+1} \|^2 + \| \nabla u_h^n \|^2 \right)
 \ge   \frac{\nu { \lambda^{-2}} \Delta t}{4}\left( \| u_h^{n+1} \|^2 + \|  u_h^n \|^2 \right)
=  \frac{\nu\Delta t}{4} \| \chi^{n+1} \|^2
\ge  \frac{\nu\Delta t C_l^2}{4} \| \chi^{n+1} \|_G^2,
\end{equation*}
and thus setting $\alpha:= \min \{ 1/2, \frac{\nu\Delta t C_l^2}{4} \}$, it holds that
\begin{equation}
\frac{\nu\Delta t}{4}\left( \| \nabla u_h^{n+1} \|^2 + \| \nabla u_h^n \|^2 \right)
+
\frac{\nu\Delta t}{2} \| \nabla u_h^{n+1} \|^2
\ge
\alpha \left(  \| \chi^{n+1} \|_G^2  + \frac{\nu\Delta t}{4} \| \nabla u_h^{n+1} \|^2 \right). \label{a4}
\end{equation}
Combining \eqref{a4} and \eqref{a2} yields
\begin{equation*}
(1+ \alpha) \left(  \| \chi^{n+1} \|_G^2  + \frac{\nu\Delta t}{4} \| \nabla u_h^{n+1} \|^2 \right)
  \le
\left( \| \chi^n \|_G^2 + \frac{\nu\Delta t}{4}\| \nabla u_h^n \|^2 \right)
+ \nu^{-1} \Delta t \| f \|_{L^{\infty}(0,\infty;V_h^*)}^2,  
\end{equation*}
which immediately implies that
\begin{multline}
 \| \chi^{n} \|_G^2  + \frac{\nu\Delta t}{4} \| \nabla u_h^{n} \|^2
  \le
\left( \frac{1}{1+\alpha} \right)^n \left( \| \chi^0 \|_G^2 + \frac{\nu\Delta t}{4}\| \nabla u_h^0 \|^2 \right)\\
+  \left( \frac{1}{1+\alpha} + ... +\left( \frac{1}{1+\alpha} \right) ^{n} \right) \nu^{-1} \Delta t \| f^{n+1} \|_{L^{\infty}(0,\infty;V_h^*)}^2.  \label{a5}
\end{multline}
Since $\alpha>0$,
\[
 \frac{1}{1+\alpha} + ... + \left( \frac{1}{1+\alpha} \right) ^{n}  = \frac{ \frac{1}{1+\alpha} - \left( \frac{1}{1+\alpha}  \right)^{n+1} }{1 - \frac{1}{1+\alpha} }
 \le
 \frac{ \frac{1}{1+\alpha} }{ \frac{\alpha}{1+\alpha} } = \frac{1}{\alpha} = \max \{ 2, \frac{4}{\nu\Delta t C_l^2} \},
 \]
and thus
\begin{multline}
\| \chi^{n} \|_G^2  + \frac{\nu\Delta t}{4} \| \nabla u_h^{n} \|^2
  \le \\
\left( \frac{1}{1+\alpha} \right)^n \left( \| \chi^0 \|_G^2 + \frac{\nu\Delta t}{4}\| \nabla u_h^0 \|^2 \right)
+
\max \{ 2\Delta t, \frac{4}{\nu C_l^2} \}
 \nu^{-1}  \| f \|_{L^{\infty}(0,\infty;V_h^*)}^2.  \label{a6}
\end{multline}
Now using the equivalence of norms for the $G$ norm and $L^2$ norm of $\chi$ completes the velocity proof.

The proof for vorticity follows identically, modulo a higher order norm on the forcing, after taking the test function to be $w_h^{n+1}$.
\end{proof}

We prove next the unconditional long-time $H^1$ stability of velocity.

\begin{theorem}[Long-time $H^1$ stability of velocity] \label{LTH1}
Let $f\in L^{\infty}(0,\infty;L^2(\Omega))$, $u_0\in H^2(\Omega)$, and set $\alpha:= \min \{ 1/2, \frac{\nu\Delta t C_l^2}{4} \}$.  Then for any $\Delta t>0$, solutions of Algorithm \ref{bdf2} satisfy for every positive integer $n$,
\begin{multline}
 C_l^2 \left(  \| \nabla u_h^{n} \|^2 + \| \nabla u_h^{n-1} \|^2 \right)  + \frac{\nu\Delta t}{4} \| A_h u_h^{n} \|^2
  \le
\left( \frac{1}{1+\alpha} \right)^n \left(  2C_u \| \nabla u_h^0 \|^2 + \frac{\nu\Delta t}{4}\| A_h u_h^0 \|^2 \right)
 \\ +
\max \left(  2\Delta t, \frac{4}{\nu C_l^2} \right)
\left( \nu^{-1}  \| f \|_{L^{\infty}(0,\infty;L^2(\Omega))}^2 + C\nu^{-3} C_5^4 C_4^2 \right) =: C_6. \label{bdf2est3}
\end{multline}
\end{theorem}
\begin{remark}\rm
Similar to the backward Euler case, the long-time $H^1$ stability bound for velocity gives $\| \nabla u_h^{n} \| \le O(Re^5)$.  If we instead bounded the nonlinear term as in the backward Euler case via
\[
|  2\Delta t ( \left(2w_h^n - w_h^{n-1}\right) \times u_h^{n+1},A_h u_h^{n+1})  |
 \le
C{\Delta t(1+}| \ln h |) \nu^{-1} C_5^2 C_4^2 { + \frac{\nu\Delta t}{2} \| A_h u_h^{n+1} \|^2},
\]
then we can get instead $\| \nabla u_h^{n} \| \le O( {(1+}| \ln h |)^{1/2} Re^3 )$.
\end{remark}

\begin{proof}
Choose $v_h=2\Delta t A_h u_h^{n+1}$ in \eqref{scheme1}, which vanishes the pressure terms, and then upper bound the forcing term  to get
\begin{multline}
\left( \| \chi^{n+1} \|_G^2 - \| \chi^n \|_G^2  \right)
+ \frac{1}{2} \| \nabla ( u_h^{n+1} - 2u_h^n + u_h^{n-1} ) \|^2
+ \nu\Delta t \| A_h u_h^{n+1} \|^2 \\
\le
\nu^{-1} \Delta t \| f^{n+1} \|^2 - 2\Delta t ( \left(2w_h^n - w_h^{n-1}\right) \times u_h^{n+1},A_h u_h^{n+1}) ,  \label{bb1}
\end{multline}
where  $\chi^{n+1} = [A_h^{1/2} u_h^{n},\ A_h^{1/2} u_h^{n+1}]^T$ and $\chi^{n} = [A_h^{1/2} u_h^{n-1},\ A_h^{1/2} u_h^{n}]^T$.
The last term on the right hand side is estimated using the same technique as in the backward Euler case from Section~\ref{s_Euler}, and then applying the $L^2$ stability estimates (which is from Theorem \ref{LTL2} in this case):
\begin{align}
|  2\Delta t ( \left(2w_h^n - w_h^{n-1}\right) \times u_h^{n+1},A_h u_h^{n+1})  |
 &\le
C \Delta t \nu^{-3} \left( \| w_h^{n} \|^4 + \| w_h^{n-1} \|^4 \right) \| u_h^{n+1} \|^2 + \frac{\nu\Delta t}{2} \| A_h u_h^{n+1} \|^2 \nonumber \\
& \le  C \Delta t \nu^{-3} C_5^4 C_4^2 + \frac{\nu\Delta t}{2} \| A_h u_h^{n+1} \|^2.
\end{align}
Combining this with \eqref{bb1}, dropping the second term on the left-hand side, and adding $\frac{\nu\Delta t}{4}\| A_h u_h^n \|^2$ to both sides produces
\begin{multline}
\left(  \| \chi^{n+1} \|_G^2  + \frac{\nu\Delta t}{4} \| A_h u_h^{n+1} \|^2 \right)
 + \frac{\nu\Delta t}{4}\left( \| A_h u_h^{n+1} \|^2 + \| A_h u_h^n \|^2 \right)
 + \frac{\nu\Delta t}{2} \| A_h u_h^{n+1} \|^2
 \\ \le
\left( \| \chi^n \|_G^2 + \frac{\nu\Delta t}{4}\| A_h u_h^n \|^2 \right)
+ \Delta t \left( \nu^{-1} \| f\|_{L^{\infty}(0,\infty;L^2(\Omega))}^2 + C\nu^{-3} C_5^4 C_4^2 \right).  \label{bb2}
\end{multline}
From here, setting $\alpha:= \min \{ 1/2, \frac{\nu\Delta t C_l^2}{4} \}$ and taking analogous steps as in the proof of the long-time $L^2$ estimate (starting from \eqref{a2}) provides us with
\begin{multline}
 \| \chi^{n} \|_G^2  + \frac{\nu\Delta t}{4} \| A_h u_h^{n} \|^2
  \le \\
\left( \frac{1}{1+\alpha} \right)^n \left( \| \chi^0 \|_G^2 + \frac{\nu\Delta t}{4}\| A_h u_h^0 \|^2 \right)
+
\max \{ 2\Delta t, \frac{4}{\nu C_l^2} \}
\left(  \nu^{-1}  \| f \|_{L^{\infty}(0,\infty;L^2(\Omega))}^2 + C \nu^{-3} C_5^4 C_4^2 \right).  \label{bb6}
\end{multline}
Finally, applying the norm equivalence for the G-norm finishes the proof.
\end{proof}

We can now prove the unconditional long-time $H^1$ stability of the vorticity.

\begin{theorem}[Long-time $H^1$ stability of vorticity] 
Let $f\in L^{\infty}(0,\infty;H^1(\Omega))$, $u_0\in H^3(\Omega)$, and set $\alpha:= \min \{ 1/2, \frac{\nu\Delta t C_l^2}{4} \}$.  Then for any $\Delta t>0$, solutions of Algorithm \ref{bdf2} satisfy for every positive integer $n$,
\begin{multline}
C_l^2 \left(  \| \nabla w_h^{n} \|^2 + \| \nabla w_h^{n-1} \|^2 \right)  + \frac{\nu\Delta t}{4} \| \Delta_h w_h^{n} \|^2
  \le
\left( \frac{1}{1+\alpha} \right)^n \left(  2C_u \| \nabla w_h^0 \|^2 + \frac{\nu\Delta t}{4}\| \Delta_h w_h^0 \|^2 \right)
 \\ +
\max \left(  2\Delta t, \frac{4}{\nu C_l^2} \right)
\left( \nu^{-1}  \| f \|_{L^{\infty}(0,\infty;H^1(\Omega))}^2 + C \nu^{-5} C_6^6 C_5^2 + \nu^{-3}C_6^4 C_5^2 \right) =: C_7. \label{bdf2est4}
\end{multline}
\end{theorem}
\begin{remark}\rm
Similar to the backward Euler case, we find that $\| \nabla w_h^n \| \le O(Re^{19})$.  However, different estimates of the nonlinear terms (i.e., using an inverse inequality as in the backward Euler case) can be used to find $\| \nabla w_h^n \| \le O({(1+}| \ln h |^{3/2} Re^{5})$.
\end{remark}

\begin{proof}
Begin by choosing $\chi_h=-2\Delta t \Delta_h w_h^{n+1}$ to get
\begin{multline}
\left( \| \chi^{n+1} \|_G^2 - \| \chi^n \|_G^2  \right)
+ \frac{1}{2} \| \nabla ( w_h^{n+1} - 2w_h^n + w_h^{n-1} ) \|^2
+ \nu\Delta t \| \Delta_h w_h^{n+1} \|^2 \\
\le
\nu^{-1} \Delta t \| \nabla \times f^{n+1} \|^2 + 2\Delta t b^*(u_h^{n+1},w_h^{n+1},\Delta_h w_h^{n+1}),  \label{cc1}
\end{multline}
where  $\chi^{n+1} = [(-\Delta_h)^{1/2} w_h^{n},\ (-\Delta_h)^{1/2} w_h^{n+1}]^T$ and $\chi^{n} = [(-\Delta_h)^{1/2} w_h^{n-1},\ (-\Delta_h)^{1/2} w_h^{n}]^T$.  Upper bounding the nonlinear term exactly as in the backward Euler case, and then using the long-time estimates proven above for the BDF2 scheme gives
\begin{align*}
| 2\Delta t b^*(u_h^{n+1},w_h^{n+1},\Delta_h w_h^{n+1}) |
& \le  C \nu^{-5} C_6^6 C_5^2 + C\nu^{-3}C_6^4 C_5^2 + \frac{\nu\Delta t}{4} \| \Delta_h w_h^{n+1} \|^2,
\end{align*}
and thus using this and dropping the second term on the left side of \eqref{cc1} yields
\begin{multline}
\left( \| \chi^{n+1} \|_G^2 - \| \chi^n \|_G^2  \right)
+ \frac{\nu\Delta t}{2} \| \Delta_h w_h^{n+1} \|^2
\le
\nu^{-1} \Delta t \| \nabla \times f^{n+1} \|^2 +C\Delta t\left( \nu^{-5} C_6^6 C_5^2 + \nu^{-3}C_6^4 C_5^2\right).  \label{cc2}
\end{multline}
From here, the same techniques as for the long-time $H^1$ stability of velocity can be used to complete the proof, modulo a higher norm on the forcing term.
\end{proof}

\section{Numerical Experiments}

We run several numerical experiments in order to test the long-time stability of Algorithm \ref{bdf2}, which is the BDF2 timestepping algorithm for the proposed velocity-vorticity method.  However, as our interest is in practical applications, we do not consider a test problem with periodic boundary conditions; instead, we consider 2D channel flow past a flat plate, which uses a Dirichlet velocity inflow, no-slip velocity on the walls, and a zero-traction outflow condition.  Thus we must appropriately modify Algorithm \ref{bdf2} so that physical boundary  conditions for the velocity and vorticity can be applied.

\begin{figure}[ht]
\includegraphics[width=0.75\textwidth]{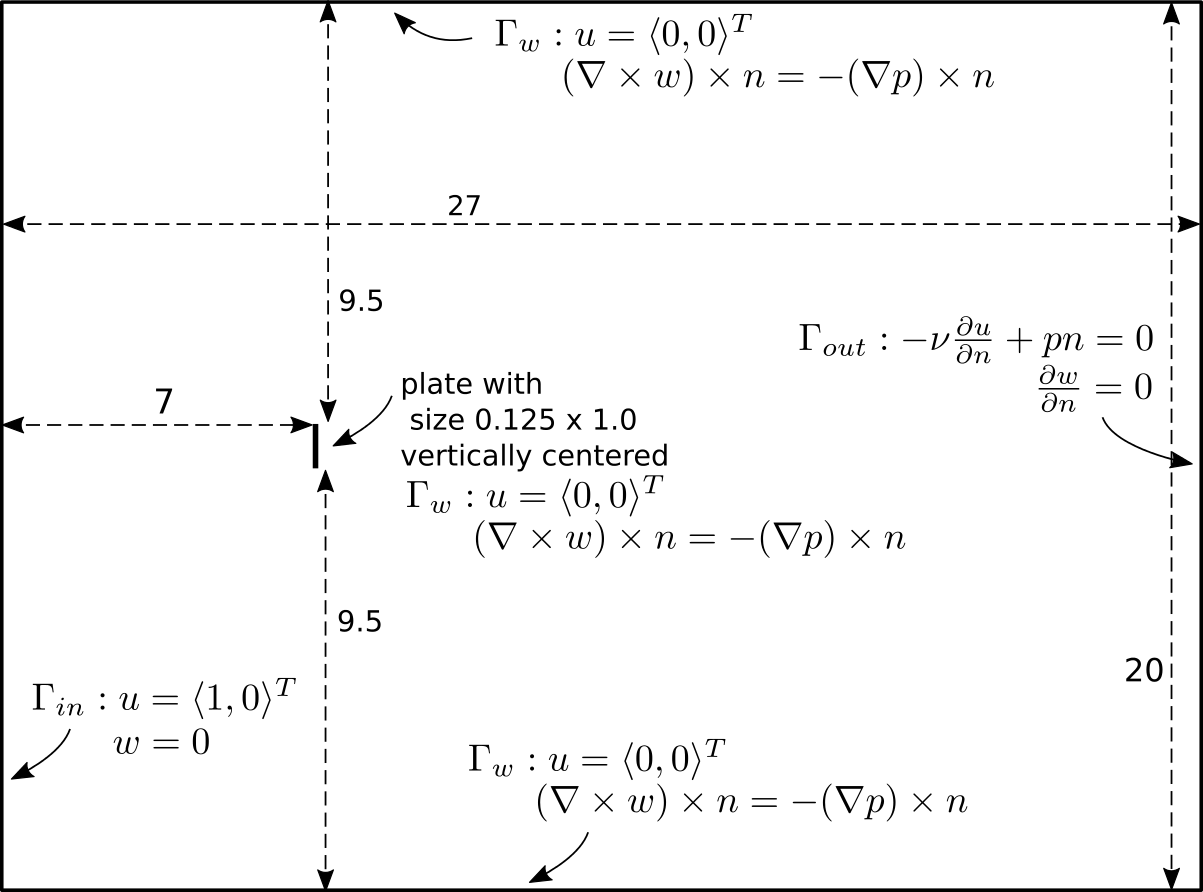}
 \caption{Setup for the flow past a normal flat plate.}
 \label{fig:plate_setup}
\end{figure}

As a numerical illustration of the long term numerical stability,  we compute the flow past a normal
flat plate following \cite{S07,S13}, see Figure~\ref{fig:plate_setup}. We take as the domain $\Omega=[-7, 20] \times [-10, 10]$, with a hole of size $0.125 \times 1$ (representing the flat plate) removed from 7 units into the channel from the left, vertically centered.  The inflow velocity is $u_{in}=\langle 1, 0 \rangle^T$, and no-slip velocity is enforced on the walls and plate.  Direct numerical simulations for this experiment are done for various Reynolds numbers $Re$, which can be considered here as $Re=\nu^{-1}$, since the length of the plate is 1, and the inflow velocity has average magnitude 1.  This
is relatively simple, but interesting problem, which resembles the flow past other bluff objects.
The  flow undergoes a first Hopf bifurcation from steady to unsteady at a relatively low Reynolds
numbers between 30 and 35 \cite{S07} and a second transition, also known as spatial transition from
two-dimensional to three-dimensional, occurs around Re=200 \cite{najjar1995simulations}.
We will test the velocity-vorticity algorithm and its long-time stability for Re=100 and Re=125.

The mathematical formulation of the problem has a constant in time non-homogeneous inflow boundary condition and zero source term. We deem this setting somewhat similar to the one analyzed in the paper (periodic boundary conditions and $L^{\infty}(0,\infty;H^1(\Omega))$-bounded right hand side), but more practically relevant.

\subsection{Velocity-vorticity formulation with boundary conditions}

Denote the domain by $\Omega$, with boundary $\partial\Omega=\Gamma_{in} \cup \Gamma_{out} \cup \Gamma_w$ split into inflow $\Gamma_{in}$, outflow $\Gamma_{out}$, and walls (of channel and plate) $\Gamma_w$.  Denote by $\tau_h$ a regular, conforming triangulation of $\Omega$.  The trial and test spaces for velocity functions are defined by
\begin{align*}
X_h^0 &:= \{ v_h \in C^0(\Omega)^2 \cap P_2(\tau_h)^2,\  \   v_h |_{\Gamma_{in} \cup \Gamma_w}=0 \}, \\
X_h^g &:= \{ v_h \in C^0(\Omega)^2 \cap P_2(\tau_h)^2,\ v_h |_{\Gamma_{in} \cup \Gamma_w}=g \},
\end{align*}
with $g=\langle 1, 0 \rangle^T$ at the inflow, $g=\langle 0, 0 \rangle^T$ on the walls, and with $P_2(\tau_h)$ denoting the
space of globally continuous functions which are quadratic on each triangle.  The discrete pressure space is taken to be
\[
Q_h = \{ q \in C^0(\Omega) \cap P_1(\tau_h) \},
\]
and the zero traction boundary condition will be enforced weakly in the formulation.  Note that $(X_h^0,Q_h)$ is the Taylor-Hood velocity-pressure element, which is known to be inf-sup stable \cite{BS08}.  The vorticity trial and test spaces are equal, since we take the vorticity at the inflow to be 0.  The outflow condition for vorticity is a homogeneous Neumann condition, which is enforced weakly by the formulation.  { The appropriate vorticity boundary condition on $\Gamma_w$ is 
\[
(\nabla \times w)  \times n=-(\nabla p)\times n
\]
where $n$ in a normal vector on $\Gamma_w$, see~\cite{GHOR15}.
In the finite element formulation this is a natural boundary condition,
resulting in the presence of the following term, cf.~\cite{GHOR15},
\[
\int_{\Gamma_w} p_h(\nabla\times\bchi_h)\cdot n\,ds
 -\int_{\partial \Gamma_w} p_h\, \bchi_h dl  \quad \forall~ \bchi_h\in W_h.
\]}
The term is added to the formulation with the known pressure from Step 1.  Thus the vorticity space is
\[
W_h  := \{ w_h \in C^0(\Omega) \cap P_2(\tau_h),\  \bw_h |_{\Gamma_{in}}=0 \}.
\]

A second modification is made to the algorithm to avoid using the Bernoulli pressure, since there is an outflow boundary.  Here, we use the identity from \cite{BL10},
\[
(\nabla \times u) \times u + \nabla \left(p + \frac12 |u|^2 \right)
=
\frac12 (\nabla \times u) \times u
+ \nabla p
+ D(u)u,
\]
where $D(u) = \frac12 \left(\nabla u + (\nabla u)^T \right)$ is the rate of deformation tensor. {Thus the
`do-nothing' conditions we use on the outflow boundary correspond to enforcing
\[
-\nu\frac{\partial u}{\partial n}+ p\,n=0,\quad \frac{\partial w}{\partial n}=0\quad\text{on}~\Gamma_{out}
\] 
in the strong formulation.}

Since there is no forcing in this test problem, we set $f=0$, and thus now Steps 1 and 2 of Algorithm \ref{bdf2} can now be written as they are computed:

Step 1:  Find $(\bu_h^{n+1},p_h^{n+1})\in (X_h^g,Q_h)$ satisfying
\begin{align*}
\frac{1}{2\Delta t}(3u_h^{n+1} - 4u_h^n + u_h^{n-1},\bv_h)
+ \frac12 \left((2\bw_h^{n}-\bw_h^{n-1}) \times \bu_h^{n+1},\bv_h\right)\qquad & \\
+ \left( D(u_h^{n+1}) (2u_h^n - u_h^{n-1}),v_h \right)
- (p_h^{n+1},\nabla \cdot \bv_h)
+ \nu(\nabla \bu_h^{n+1},\nabla \bv_h) & = 0\quad \forall~\bv_h \in X_h^0 \\
(\nabla \cdot \bu_h^{n+1},q_h)& = 0 \quad \forall~q_h \in Q_h.
\end{align*}

Step 2: Find $w_h^{n+1} \in W_h$ satisfying
\begin{multline}\label{step2}
\frac{1}{2\Delta t}(3\bw_h^{n+1} - 4\bw_h^n + w_h^{n-1},\bchi_h)
+ (\bu_h^{n+1} \cdot \nabla \bw_h^{n+1},\bchi_h)
+ \nu(\nabla \bw_h^{n+1},\nabla \bchi_h)
\\
 =
 - \int_{\Gamma_w} p_h^{n+1}(\nabla\times\bchi_h)\cdot n\,ds
 +\int_{\partial \Gamma_w} p_h^{n+1}\, \bchi_h dl  \quad \forall~ \bchi_h\in W_h.
\end{multline}
We note that since globally continuous pressure elements are used, the right hand side of \eqref{step2} can be equivalently written as
\[
- \int_{\Gamma_w} p_h^{n+1}(\nabla\times\bchi_h)\cdot n\,ds
 +\int_{\partial \Gamma_w} p_h^{n+1}\, \bchi_h dl
 =
  \int_{\Gamma_w} (\nabla p_h^{n+1} \times n) \cdot \bchi_h \,ds.
 \]

\subsection{Channel flow past a flat plate at Re=100 and Re=125}\label{s_exp}

The BDF2 velocity-vorticity scheme was computed for both Re=100 and Re=125 ($\nu$=Re$^{-1}$), using 3 Delaunay generated triangular meshes which provided 79509 total degrees of freedom (dof), 116045 dof, and 159055 dof with the $(P_2^2,P_1,P_2)$ velocity-pressure-vorticity elements.  The simulations started the flow from rest $(u_h^0=0)$, and were run to an endtime T=200.  For each mesh, several timestep choices were made, starting with $\Delta t$=0.04, and then cutting $\Delta t$ in half until convergence (i.e., successive solutions' statistics matched).  For both Re=100 and Re=125, the smallest $\Delta t$ was 0.01.

Quantities of interest for this problem is the long-time average of the drag coefficient $C_d$, and the Strouhal number.  The Strouhal number was calculated as in \cite{S07,S13}, using the fast Fourier transform of the transverse velocity at (4.0, 0.0) from T=120 to T=200.  The drag coefficients are defined at each $t^n$ to be
\begin{align*}
C_d(t^m) & =  \frac{2}{\rho L U_{max}^2} \int_{S} \left( \rho \nu \frac{\partial u_{t_S}(t^m)}{\partial n} n_y - p_h^m n_x \right) \ dS,
\end{align*}
where $S$ is the plate, $\bn=\langle n_x,n_y \rangle$ is the outward normal vector to $S$ pointing into the domain, $\bu_{t_S}(t^m)$ is the tangential velocity of $u_h^m$, the density $\rho=1$, the max velocity at the inlet $U_{max}=1$, and $L=1$ is the length of the plate.  The integral is calculated by transforming it into a global integral, which is believed to be more accurate \cite{J04}.  The results for time averaged $C_d$ and the Strouhal numbers from the simulations for each Re, and for each mesh (with $\Delta t=0.01$), are shown in Table~\ref{dragtable}, along with reference values taken from \cite{S13}.  We observe that the 116K dof mesh and the 159K dof meshes agree well with the reference values at Re=100 and Re=125.   It appears we have achieved (or are close to) grid-convergence, and we note that for the Strouhal number, since the FFT was used with 8,000 timesteps, 0.177 was the closest discrete frequency value to 0.174, and 0.189 was the next biggest discrete value compared to 0.183.   We also plot the time-averaged vorticity in Figure~\ref{plateplots3}, and instantaneous velocity (as speed contours) in Figure~\ref{plateplots4}; both plots match the reference plots given in \cite{S13}.

\begin{table}[h!]
\begin{center}
\begin{tabular}{|c|c|c|c|c|c|}
\hline
Method & Mesh & Re &  $C_d$ & Strouhal number \\ \hline
Vel-Vort & 78K dof & 100 & 2.48 & 0.195 \\ \hline
Vel-Vort    & 116K dof & 100 & 2.59 &  0.189  \\ \hline
Vel-Vort   & 159K dof & 100 & 2.58 &  0.189  \\ \hline
Saha \cite{S13} &   & 100 & 2.60  & 0.183 \\ \hline
  & & & &  \\ \hline
 Vel-Vort & 78K dof  & 125 & 2.57 & 0.189  \\ \hline
Vel-Vort  & 116K dof  & 125  & 2.60 &  0.177 \\ \hline
Vel-Vort &  159K dof & 125 & 2.59 & 0.177 \\ \hline
Saha \cite{S13} &   & 125 & 2.55 &  0.174 \\ \hline
\end{tabular}
\end{center}
\caption{\label{dragtable}  Shown above are Strouhal numbers and long-time average drag coefficients for solutions on varying meshes, for Re=100 and Re=125.  Reference values are also given for comparison.}
\end{table}

Also of interest is the stability of computed solutions in the $\| u_h^n \|_{L^2},\ \| u_h^n \|_{H^1},\ \| w_h^n \|_{L^2},\ \| w_h^n \|_{H^1}$ norms versus time $t^n$, since we proved in Section~\ref{s_BDF} that these norms are all long-time stable (at least, in the periodic setting), independent of the timestep $\Delta t$ and mesh width $h$.  Plots of these norms versus time are shown for Re=100 in Figure~\ref{plateplots1} and for Re=125 in Figure~\ref{plateplots2} for varying timesteps. Each norm appears to be long-time stable.  Moreover, we do not observe the very large scaling of any of the norms with $Re$.  Although $\| \nabla w_h^n\| \approx O(500)$ is an order of magnitude larger than $\| w_h^n\|$, it is still a very reasonable size and nowhere near $O(Re^{19})$ or even $O({ (1+|\ln h|)}^{\frac32} Re^{5})$.

\begin{figure}[h!]
\label{dec}
\begin{center}
Re=100 \\
\includegraphics[width=0.9\textwidth,height=0.2\textwidth, trim=100 10 60 0, clip]{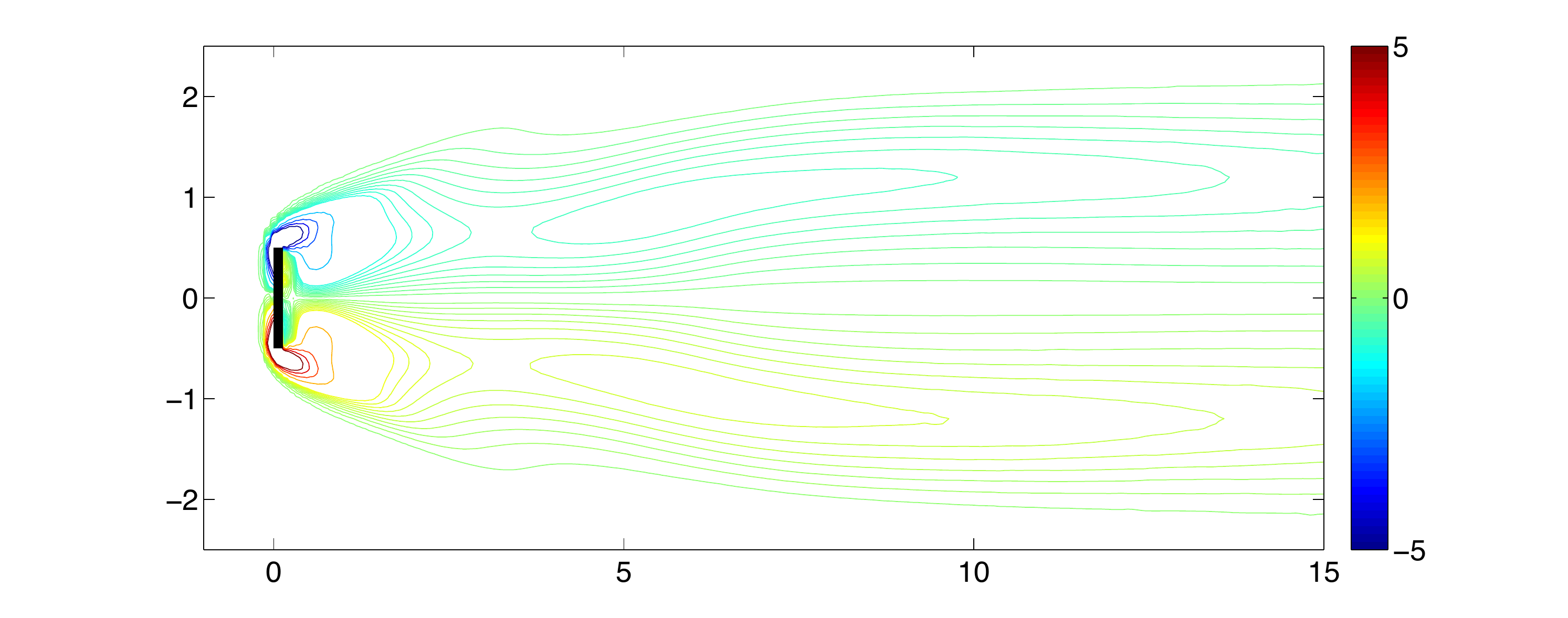} \\
Re=125\\
\includegraphics[width=0.9\textwidth,height=0.2\textwidth, trim=100 10 60 0, clip]{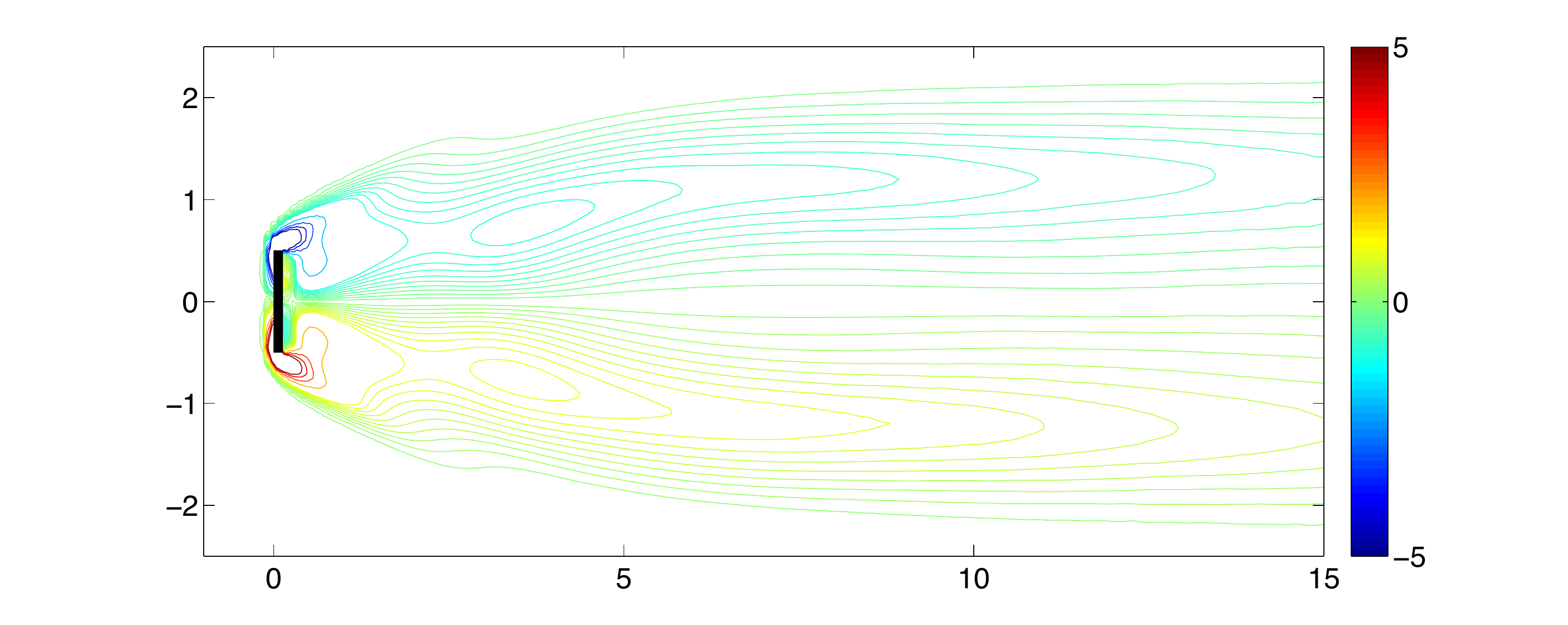}
\end{center}
\caption{\label{plateplots3} Shown above are plots of the time-averaged vorticity contours.}
\end{figure}

\begin{figure}[h!]
\label{dec2}
\begin{center}
Re=100 \\
\includegraphics[width=0.9\textwidth,height=0.2\textwidth, trim=100 10 60 0, clip]{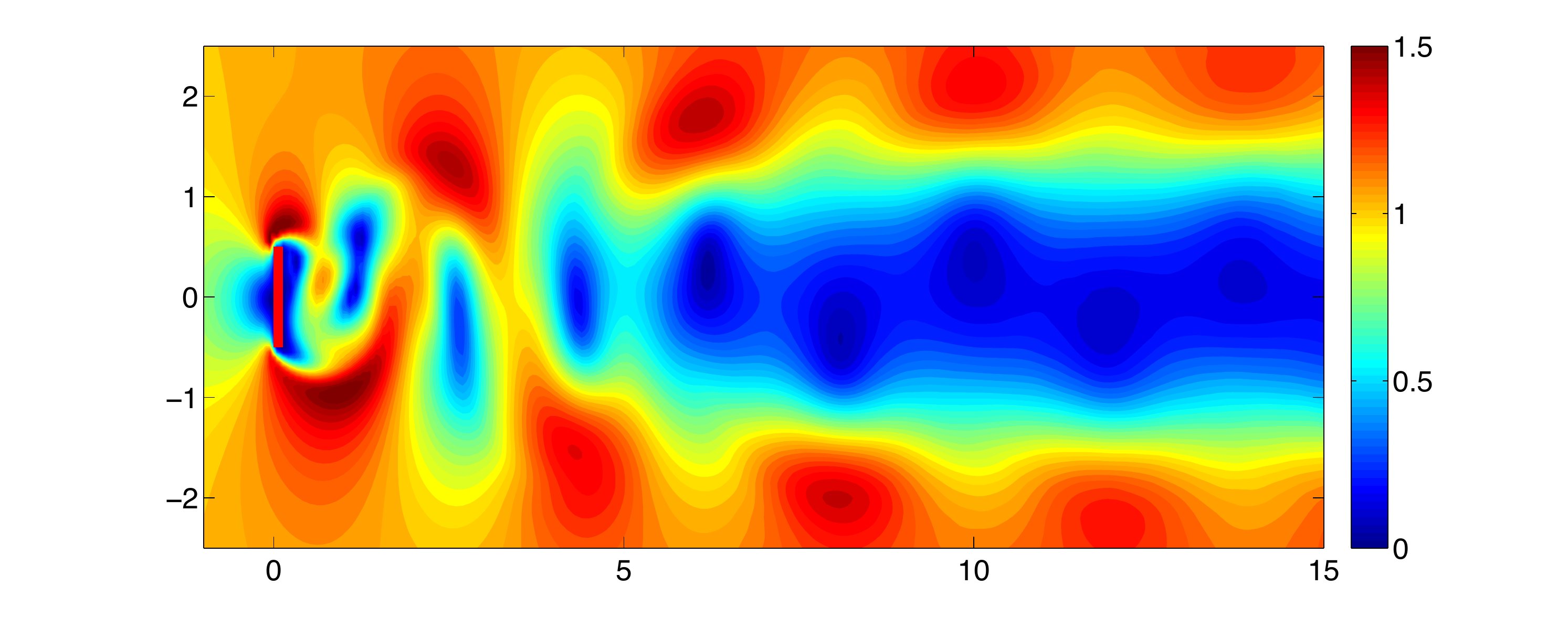} \\
Re=125\\
\includegraphics[width=0.9\textwidth,height=0.2\textwidth, trim=100 10 60 0, clip]{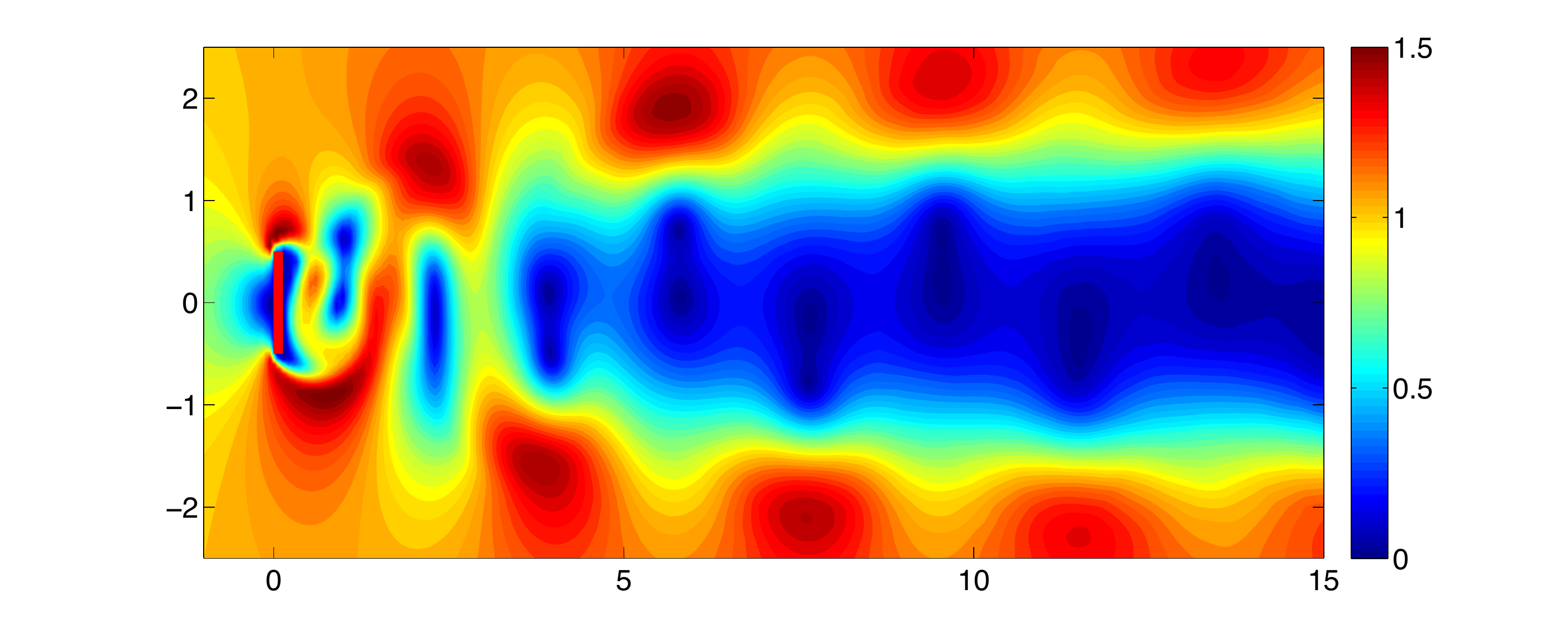}
\end{center}
\caption{\label{plateplots4} Shown above are plots of the speed contours of the velocity solutions at T=200.}
\end{figure}

\begin{figure}[h!]
\label{dec3}
\begin{center}
\includegraphics[width=0.49\textwidth,height=0.25\textwidth, trim=20 0 60 0, clip]{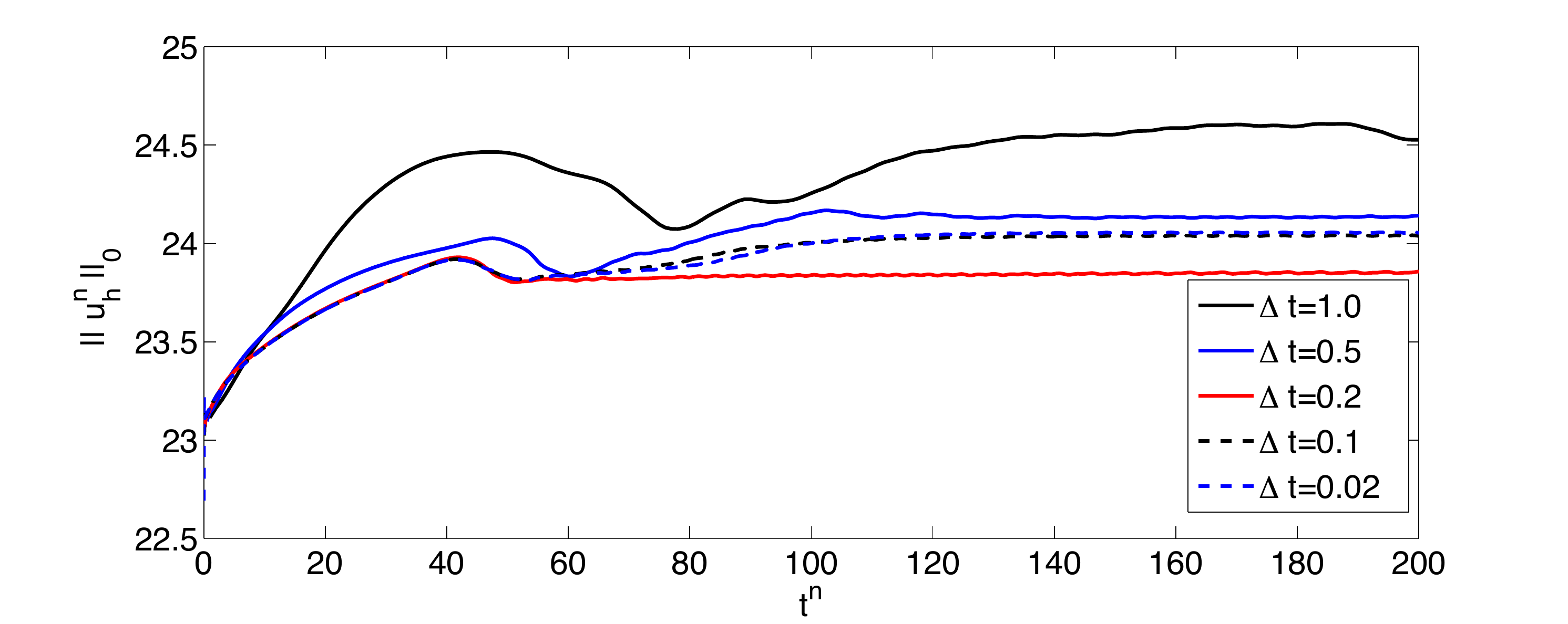}
\includegraphics[width=0.49\textwidth,height=0.25\textwidth, trim=20 0 60 0, clip]{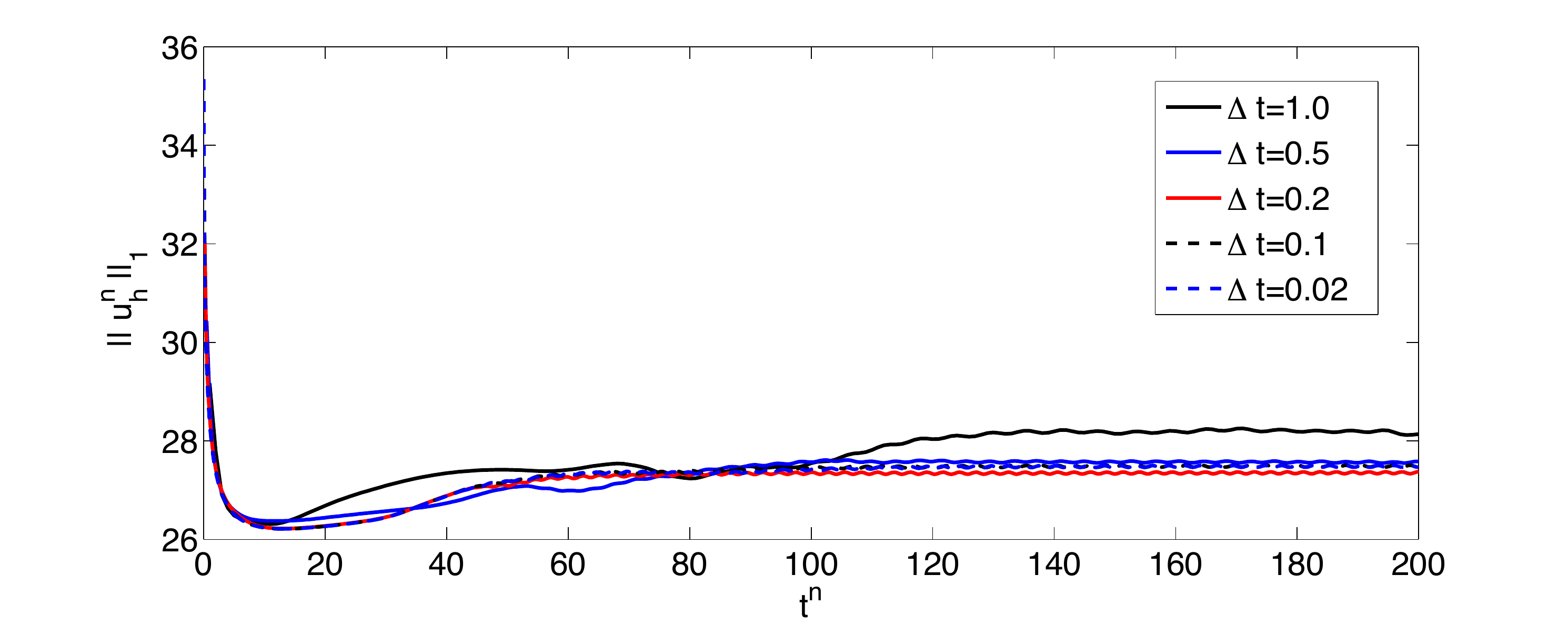}\\
\includegraphics[width=0.49\textwidth,height=0.25\textwidth, trim=20 0 60 0, clip]{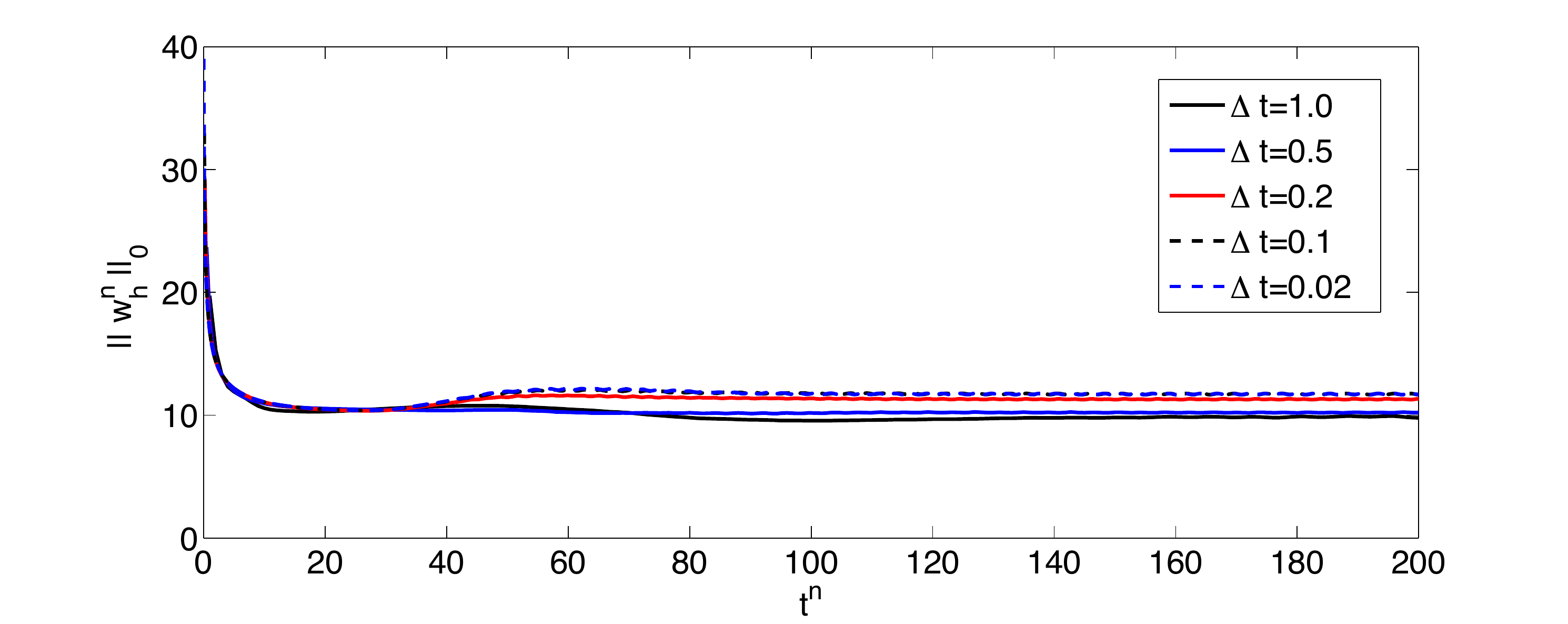}
\includegraphics[width=0.49\textwidth,height=0.25\textwidth, trim=20 0 60 0, clip]{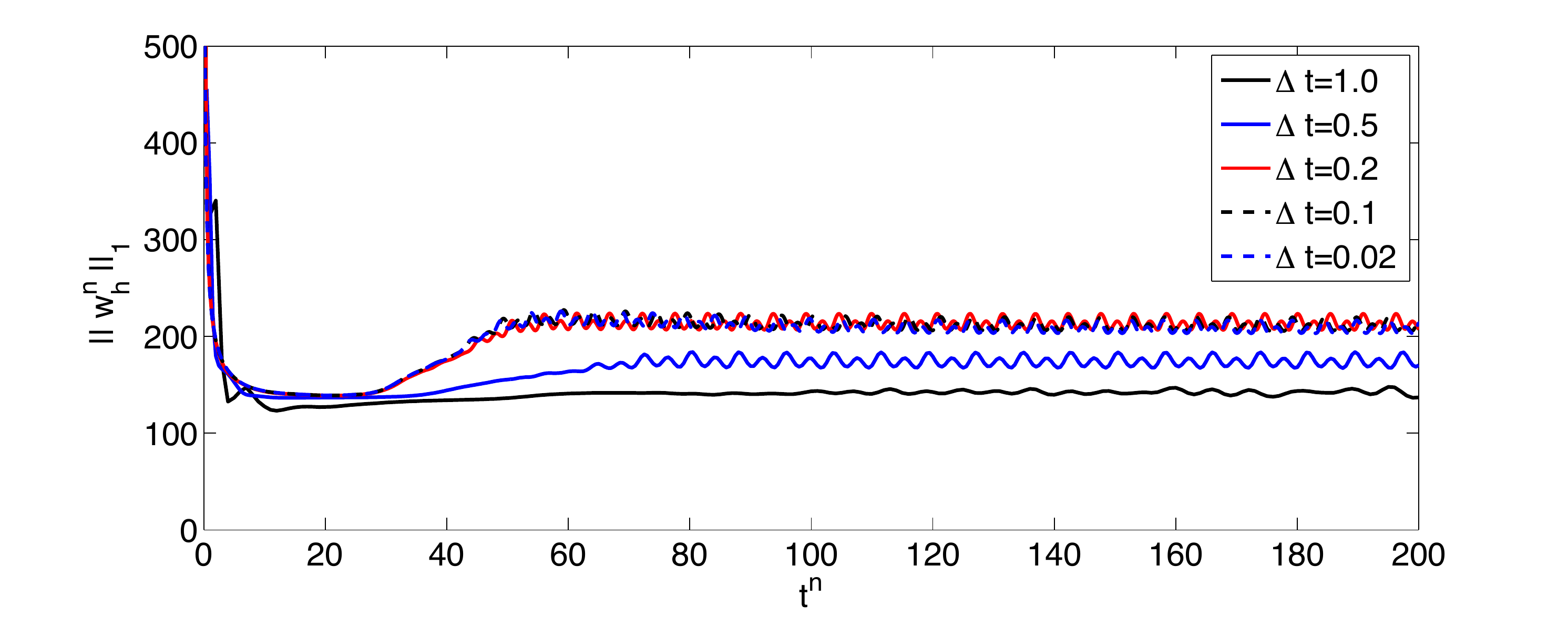}
\end{center}
\caption{\label{plateplots1} Shown above are plots of the Re=100 solution norms versus time, found using Mesh 3 (the finest mesh). }
\end{figure}

\begin{figure}[h!]
\label{dec4}
\begin{center}
\includegraphics[width=0.49\textwidth,height=0.25\textwidth, trim=20 0 60 0, clip]{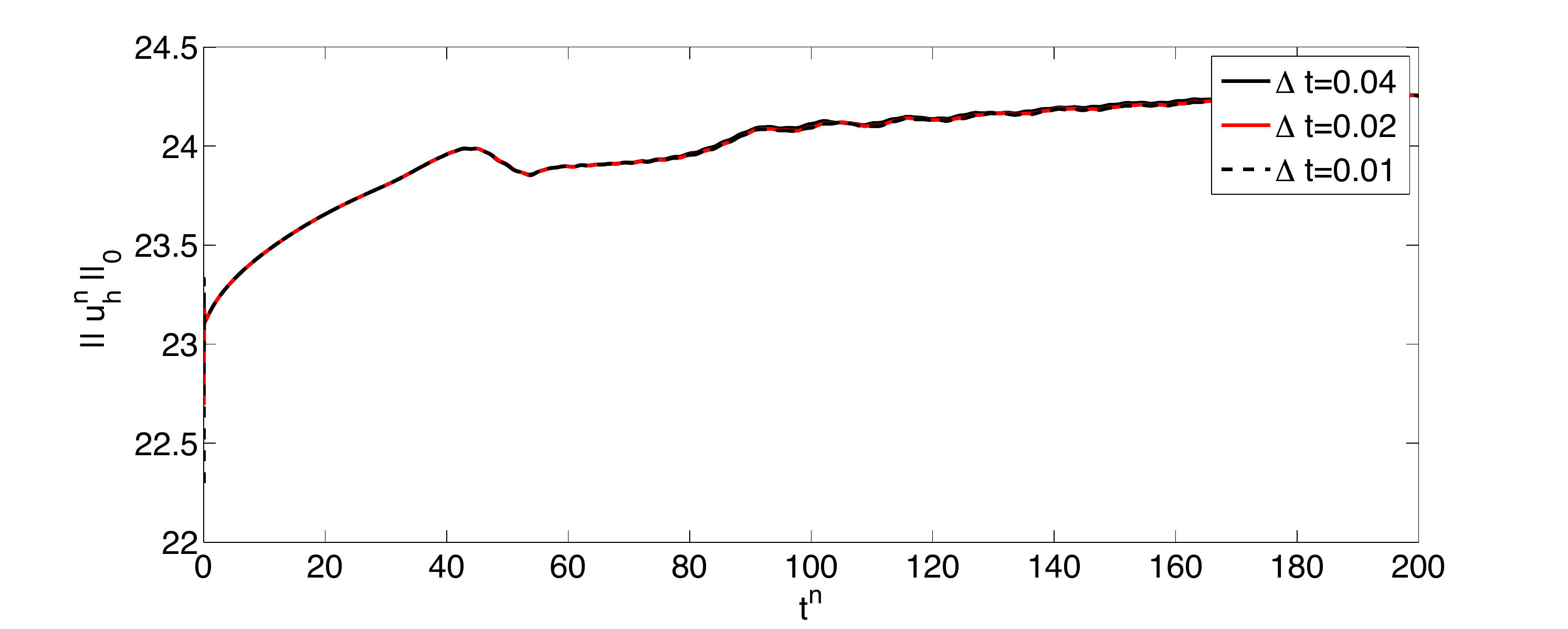}
\includegraphics[width=0.49\textwidth,height=0.25\textwidth, trim=20 0 60 0, clip]{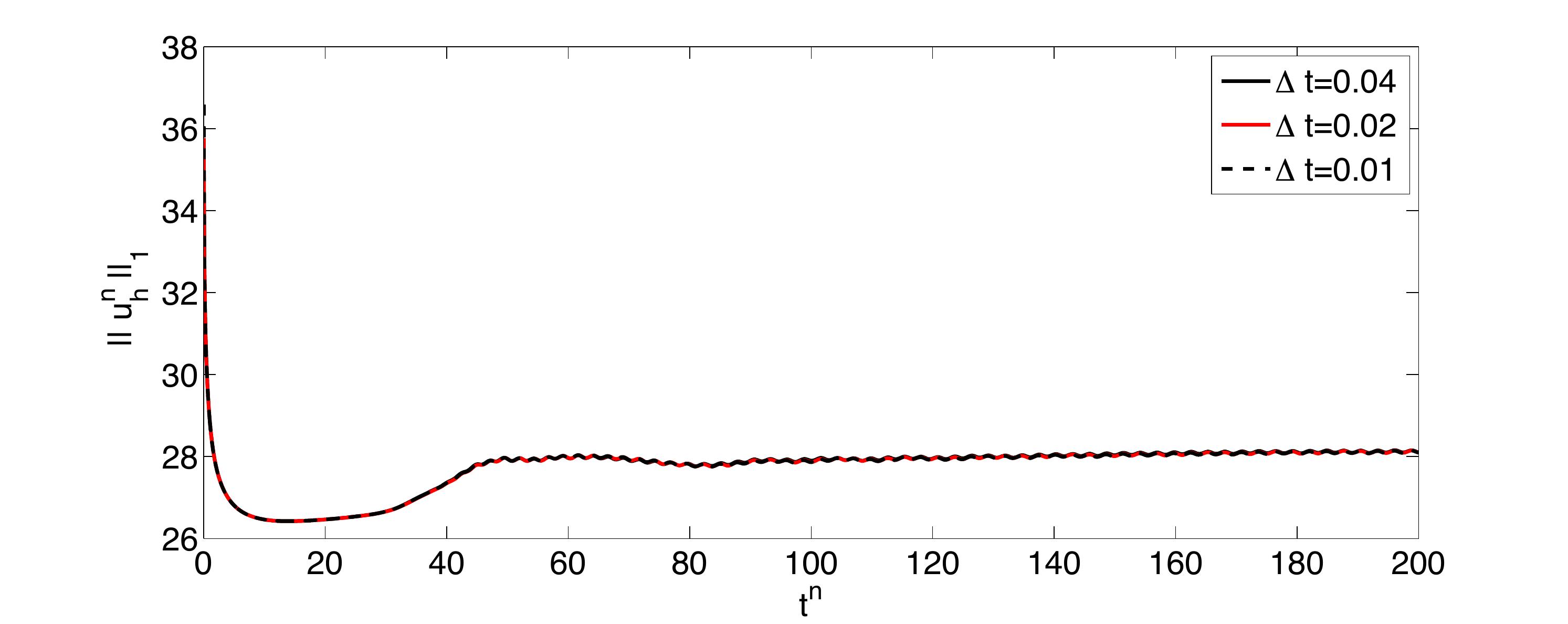} \\
\includegraphics[width=0.49\textwidth,height=0.25\textwidth, trim=20 0 60 0, clip]{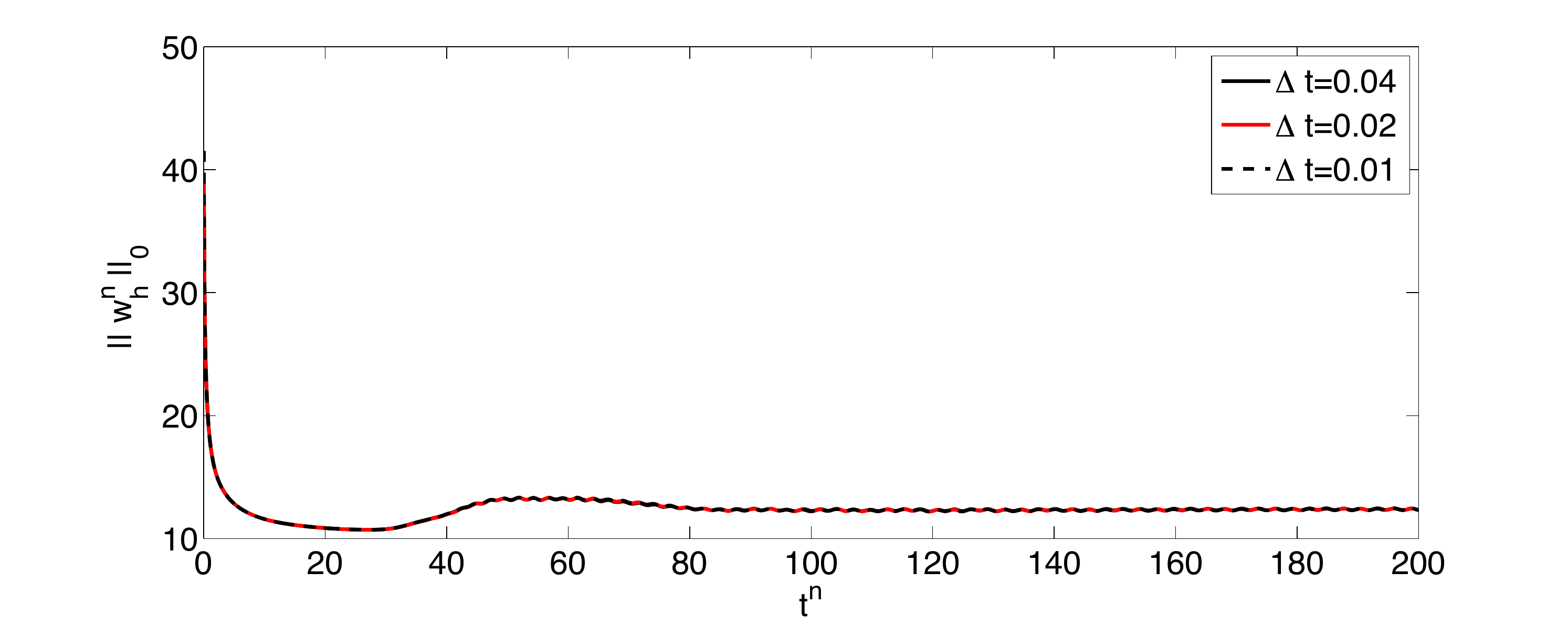}
\includegraphics[width=0.49\textwidth,height=0.25\textwidth, trim=20 0 60 0, clip]{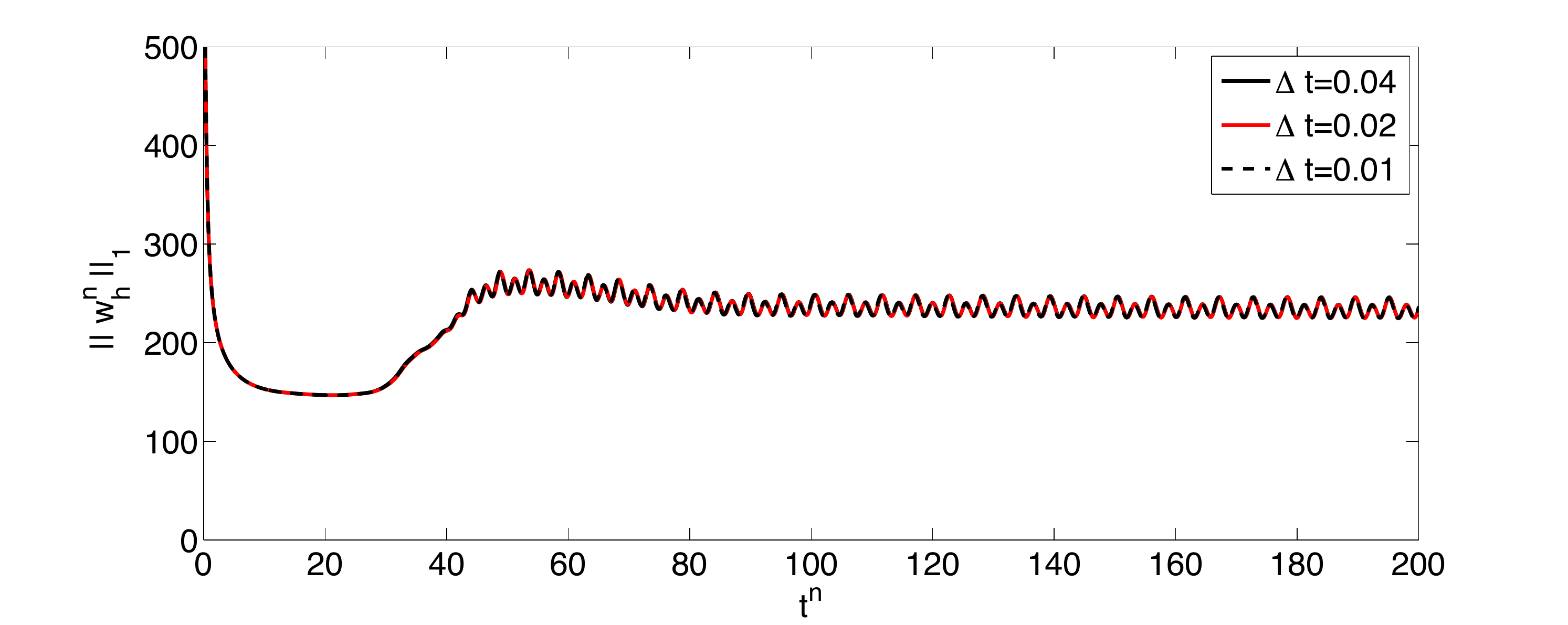}
\end{center}
\caption{\label{plateplots2} Shown above are plots of the Re=125 solution norms versus time, found using Mesh 3 (the finest mesh).}
\end{figure}

\section{Conclusions and Future Directions}
We have proven unconditional long-time stability of a scheme based on a velocity-vorticity formulation, and a finite-element-in-space BDF2-in-time IMEX discretization for the 2D Navier-Stokes equations.  Long-time stability was proven in both the $L^2$ and $H^1$ norms for both velocity and vorticity, and the estimates hold for any $\Delta t>0$.  The scheme is non-standard, and so we tested it on a benchmark problem on flow past a flat plate; it performed very well.

It would be interesting to study Algorithm \ref{bdf2}, and variations thereof, for 3D flows.  The difference in 3D is that the vortex stretching term $-(w\cdot\nabla u)$ appears in the vorticity equation.  Since the 2D algorithm is proven herein to be unconditionally long-time stable, any instability in the 3D algorithm can be immediately attributed to the vortex stretching term and/or its numerical treatment.  Isolating this behavior may give insight into better stabilization methods for higher Reynolds number flows in 3D.

\bibliographystyle{plain}

\bibliography{graddiv,bibliography}

\end{document}